\documentclass[12pt,reqno]{amsart}
\AtBeginDocument{\addtocontents{toc}{\protect\setlength{\parskip}{0pt}}}

\usepackage{latexsym, amsmath,amssymb}
\usepackage{hyperref}
\usepackage{tikz}
\usetikzlibrary{calc,trees,shadows,positioning,arrows,chains,shapes.geometric,
decorations.pathreplacing,decorations.pathmorphing,shapes,
matrix,shapes.symbols,patterns,intersections,fit}

\theoremstyle{plain}
\newtheorem{thm}{Theorem}[section]
\newtheorem{lem}[thm]{Lemma}
\newtheorem{prop}[thm]{Proposition}
\newtheorem{cor}[thm]{Corollary}

\theoremstyle{definition}
\newtheorem{defi}[thm]{Definition}
\newtheorem{example}[thm]{Example}

\newtheorem{rem}[thm]{Remark}
\numberwithin{equation}{section}

	\newcommand{\eps}{\varepsilon}
		
				\newcommand{\gs}{\gtrsim}
			\newcommand{\la}{\lambda}
			\newcommand{\1}{{\rm 1\hspace*{-0.4ex}%
					\rule{0.1ex}{1.52ex}\hspace*{0.2ex}}}

\begin{document}

\title[Orthonormal Strichartz estimates on manifolds]{Strichartz estimates for orthonormal systems on compact manifolds: the non-sharp region}

\date{\today. \qquad \emph{2020 AMS MSC.}    42B37, 58J50, 35P05, 35R01, 35S30}
	


\author{Hongzhou Ji, Liping Xu and An Zhang}

\address{School of Mathematical Sciences, Beihang University, Beijing, 100191, PR China}
\email{jhz2509@buaa.edu.cn}
\email{xuliping.p6@gmail.com}
\email{anzhang@buaa.edu.cn}
\maketitle
\begin{abstract}
We establish new Strichartz estimates for orthonormal systems on compact Riemannian manifolds in the non-sharp admissible region of exponents, covering wave, Klein-Gordon, and fractional Schr\"odinger equations. Our approach combines the result of Wang-Zhang-Zhang \cite{wang2025strichartz} on the sharp admissible line with a Lieb-Sobolev inequality derived from a recent Cwikel estimate due to Sukochev-Yang-Zanin \cite{sukochev2025singular}, along with an alternative globalization method based on localized weak Lorentz estimates. Our results extend the Euclidean results of Bez-Hong-Lee-Nakamura-Sawano \cite{bez2019strichartz} and Bez-Lee-Nakamura \cite{bez2021strichartz}, as well as the classical single-function estimates on manifolds due to Kapitanski \cite{kapitanski1989some}, Burq-G\'erard-Tzvetkov \cite{MR2058384}, and Dinh \cite{dinh2016strichartz}.
\end{abstract}
\tableofcontents  

\section{Introduction}
\label{int}

Let $(M,g)$ be a $d$-dimensional smooth compact Riemannian manifold without boundary, with $d\geq 1$, and let 
$\Delta = -\Delta_g$, where $\Delta_g$ denotes the standard Laplace–Beltrami operator on $M$. Let $P$ be a dispersive operator.
The classical (single-function) Strichartz estimate for the solution $e^{itP}f$ to the following initial value problem
\begin{equation*}
    \left\{\begin{array}{l}
        i\partial_t u + P u = 0, \quad (x,t) \in M \times \mathbb R,\\
        u(\cdot,0) = f,
    \end{array}\right.
\end{equation*} 
has important applications in the well-posedness analysis of this equation, which reduces to the fractional Schr\"odinger equation when $P = \Delta^{\alpha/2}$ ($\alpha \neq 0,1$), to the wave equation when $P = \sqrt{\Delta}$, and to the Klein–Gordon equation when $P = \sqrt{1+\Delta}$. In this paper, we mainly consider \emph{orthonormal Strichartz estimates} of the form
\begin{equation}\label{eq:str}
	\Big\|\sum_j \nu_j |e^{itP} f_j|^2\Big\|_{L_t^{p/2} L_x^{q/2}(I \times M)} \lesssim \|\nu\|_{\ell^\beta}
\end{equation}
for orthonormal systems $(f_j)_j$ in a suitable Hilbert space, and complex sequences $\nu = (\nu_j)_j$ in $\ell^\beta$, where $I \subset \mathbb R$ is a bounded interval.
This orthonormal inequality describes the  evolution of large-scale quantum systems.
In order to facilitate the analysis, we first introduce the full \emph{admissible} Lebesgue exponents.

\begin{defi}
    Let $\sigma > 0$. A pair of exponents $(p,q) \in [2,\infty] \times [2,\infty)$ is called $\sigma$-admissible if 
 \begin{equation}\label{eq:exp-fulladm}
        \frac{1}{p} + \frac{\sigma}{q} \leq \frac{\sigma}{2}.
    \end{equation}
    Furthermore, $(p,q)$ is called \emph{non-sharp} $\sigma$-admissible if
 \begin{equation}\label{eq:exp-nonsharp}
 \frac{1}{p} + \frac{\sigma}{q} < \frac{\sigma}{2},\end{equation}
    and \emph{sharp} $\sigma$-admissible if
 \begin{equation}\label{eq:exp-sharp}
 \frac{1}{p} + \frac{\sigma}{q} = \frac{\sigma}{2}.\end{equation}
\end{defi}

In Figure \ref{fig1}, when $\sigma = d/2$, the line segment $BC$ corresponds to the \emph{sharp admissible line}, while the \emph{non-sharp admissible region} corresponds to the quadrilateral $OBCD$ excluding the segments $OD$ and $BC$. We will establish sharp orthonormal Strichartz estimates for \emph{non-sharp} admissible exponents on general compact manifolds.

\vskip2mm
\noindent\textbf{Notations.} Throughout this paper, we use Vinogradov's notation $X \lesssim_a Y, X \approx_a Y$. We write $\langle \xi \rangle = \sqrt{1+|\xi|^2}$, and similarly $\langle D \rangle$ for operator calculus.

\subsection{Some known results}
The Euclidean Strichartz estimates are relatively well-developed. 
The classical estimate for single functions was first established by Strichartz \cite{strichartz1977restrictions} for the wave equation, and by Ginibre--Velo \cite{ginibre1992smoothing} for the Schr\"odinger equation, in the non-endpoint cases.
Keel--Tao \cite{keel1998endpoint} established a general abstract framework for Strichartz estimates and obtained the full region of admissible exponent pairs, including the previously inaccessible endpoint cases. 
Guo--Peng \cite{MR2353675} obtained an endpoint Strichartz estimate for the kinetic transport equation. 
Guo--Peng--Wang \cite{guo2008decay} studied the decay of a class of wave equations in a unified way. 
See also \cite{MR1101221,MR2787438,MR3286047} for fractional or radial cases and applications.

The study of functional inequalities for orthonormal systems dates back to Lieb--Thirring \cite{LT}, and these inequalities play a crucial role in quantum mechanics.
Recently, in the Schr\"odinger setting, Frank--Lewin--Lieb--Seiringer \cite{frank2014strichartz} extended the classical single-function Strichartz estimates to orthogonal families in $L^2(\mathbb{R}^d)$.
Frank--Sabin \cite{frank2017restriction} later extended these estimates to the full sharp admissible line via a duality principle and complex interpolation for the Schr\"odinger, wave, and Klein--Gordon equations. 
Bez--Hong--Lee--Nakamura--Sawano \cite{bez2019strichartz}
considered initial data with Sobolev regularity for non-sharp admissible exponents for the Schr\"odinger propagator. They developed approaches based on Lorentz
norms and localized estimates. 
Later, Bez--Lee--Nakamura \cite{bez2021strichartz} introduced weighted oscillatory integral estimates
and extended the scope of orthonormal Strichartz estimates to the full admissible region and to the wave, Klein--Gordon, and fractional Schr\"odinger equations in Euclidean spaces.  Feng--Mondal--Song--Wu \cite{feng2026orthonormal} investigated a wide class of dispersive semigroups.

On compact Riemannian manifolds, single-function Strichartz estimates were obtained by Kapitanski \cite{kapitanski1989some} for the wave equation, by Burq--G\'erard--Tzvetkov \cite{MR2058384} for the Schr\"odinger equation, by Dinh \cite{dinh2016strichartz} for the fractional Schr\"odinger equation, and by Cacciafesta--Danesi--Meng \cite{cacciafesta2024strichartz} for the wave, Klein--Gordon and Dirac equations. See also the interesting works \cite{MR1101221,MR1209299,MR3273490,MR3374964,MR3692323,MR4436143} for the periodic setting.

\begin{thm}\label{sig0}
    Let $d\geq 1$ and $I\subset \mathbb{R}$ be a bounded interval. 
    For $\frac{d}{2}$-admissible pairs $(p,q)$, and $\alpha \in (0,\infty)\setminus \{1\}$, define  
    \begin{align}\label{eq:gama}
        \gamma_{\alpha}(p,q)=\begin{cases}
            \frac{d}{2}-\frac{d}{q}-\frac{1}{p}, & \alpha>1,\\[4pt]
            \frac{d}{2}-\frac{d}{q}-\frac{\alpha}{p}, & \alpha\in(0,1).
        \end{cases}
    \end{align}
    Then for all $f\in H^{\gamma_{\alpha}(p,q)}(M)$, 
    \[
        \|e^{it\Delta^{\alpha/2}} f\|_{L_t^p L_x^q(I\times M)} \lesssim \|f\|_{H^{\gamma_{\alpha}(p,q)}(M)}.
    \]
\end{thm}

\begin{thm}\label{sig1}
    Let $d\geq 2$, $m\geq 0$ and $I\subset \mathbb{R}$ be a bounded interval.
    For $\frac{d-1}{2}$-admissible pairs $(p,q)$, define  
    \begin{align}\label{eq:gam}
        \gamma(p,q)=\frac{d}{2}-\frac{d}{q}-\frac{1}{p}.
    \end{align}
    Then for all $f\in H^{\gamma(p,q)}(M)$,  
    \[
        \|e^{it \sqrt{\Delta+m^2}} f\|_{L_t^p L_x^q(I\times M)} \lesssim \|f\|_{H^{\gamma(p,q)}(M)}.
    \]
\end{thm}

\begin{rem}
    Different from the Euclidean case, further loss of derivatives occurs naturally in the compact case. Besides, while the regularity exponent $\gamma(p,q)$ in the wave case shares the same functional form as $\gamma_\alpha(p,q)$ (for $\alpha>1$) in the Schr\"odinger case, the admissibility conditions imposed on $(p,q)$ are distinct. Explicitly, on the sharp line and in terms of $p$,  
    \begin{align*}
        \gamma_\alpha(p,q) &= \begin{cases}
           \quad \dfrac{1}{p}\,, & \alpha>1,\\[11pt]
            \dfrac{2-\alpha}{p}\,, & \alpha\in(0,1);
        \end{cases} \qquad
        \gamma(p,q) = \frac{1}{p} \cdot \frac{d+1}{d-1}.
    \end{align*}
\end{rem}

In contrast with the classical single-function case, less is known for orthonormal systems on compact manifolds.
Spectral cluster estimates for orthonormal families were obtained by Frank–Sabin \cite{frank2017restriction}, with further refinements on manifolds of non-positive curvature by Ren–Zhang \cite{ren2024improved}.
For the flat torus, Nakamura \cite{MR4068269} derived Strichartz estimates for orthonormal systems using frequency-global dispersive estimates of Kenig–Ponce–Vega \cite{MR1101221}.
Subsequently, Wang–Zhang–Zhang \cite{wang2025strichartz} established sharp Strichartz estimates for orthonormal systems on general compact manifolds for the wave, Klein–Gordon, and fractional Schrödinger equations, combining frequency-localized dispersive estimates, the duality principle, and (vector-valued) Littlewood–Paley theory. Jian–Wang–Xi \cite{jian2025sharp} extended the restriction bounds of Burq–Gérard–Tzvetkov \cite{MR2322684} to arbitrary systems of spectral clusters.

To facilitate the presentation, we introduce some additional exponents. For $\sigma \geq \frac{1}{2}$ and a $\sigma$-admissible pair $(p,q)$, we define $\beta_\sigma(p,q)$ by
\begin{equation}\label{eq:beta}
\frac{\sigma}{\beta_\sigma(p,q)} = \frac{1}{p} + \frac{2\sigma}{q}.\end{equation}
We also introduce the trivial point $B$ and the bad point $D$ given by
\[
B = \left(\frac{1}{2}, 0\right), \qquad D = \left(0, \frac{1}{2}\right),
\]
and the critical point $A_\sigma$ and the diagonal point $E_\sigma$ given by
\begin{equation}\label{eq:AE}
A_\sigma = \left(\frac{2\sigma-1}{2(2\sigma+1)}, \frac{\sigma}{2\sigma+1}\right), \qquad
E_\sigma = \left(\frac{\sigma}{2(\sigma+1)}, \frac{\sigma}{2(\sigma+1)}\right).
\end{equation}
If $\sigma > 1$, we also introduce the Keel–Tao endpoint
\[
C_\sigma = \left(\frac{\sigma-1}{2\sigma}, \frac{1}{2}\right).
\]
In our considerations, where $\sigma = \frac{d}{2}$ or $\sigma = \frac{d-1}{2}$, the admissible pairs $(p,q)$ are shown in Figures~\ref{fig1} and~\ref{fig2}, respectively.

We present below the orthonormal Strichartz estimate for the Schr\"odinger propagator on compact manifolds. 
In the exponent plane of $(1/q,1/p)$, for two points $A$ and $B$, we denote by $(A,B]$ the segment connecting $A$ and $B$, with the obvious meaning at the endpoints.

\begin{thm}\label{thm:sharp}
Let $d\ge 1$, $\alpha \in (0,\infty)\setminus\{1\}$, $P = \Delta^{\alpha/2}$, and $I\subset\mathbb{R}$ be a bounded interval with $|I|\approx 1$. For sharp $\frac{d}{2}$-admissible pairs $(p,q)$ satisfying \eqref{eq:exp-sharp} with $\sigma=d/2$, let $s > \gamma_{\alpha}(p,q)$ with $\gamma_{\alpha}(p,q)$ given by \eqref{eq:gama} and let $\beta \ge 1$ satisfy:
\begin{itemize}
\item[(i)] If $(\frac{1}{q},\frac{1}{p}) \in (A_{d/2}, B]$, then $\beta \le \beta_{\frac d2}(p,q)= \frac{d}{d-2/p} = \frac{2q}{q+2}$.
\item[(ii)] If $(\frac{1}{q},\frac{1}{p}) \in (C_{d/2}, A_{d/2}]$, then $\beta < p/2$.
\item[(iii)] If $(\frac{1}{q},\frac{1}{p}) = C_{d/2}$, then $\beta = 1$.
\end{itemize} 
Then the Strichartz estimate \eqref{eq:str} holds for all orthonormal systems $(f_j)_j$ in $H^s(M)$ and all  $\nu=(\nu_j)_j$ in $\ell^\beta(\mathbb C)$.
\end{thm}

\begin{figure}[h]
\centering
\includegraphics[width=0.8\textwidth]{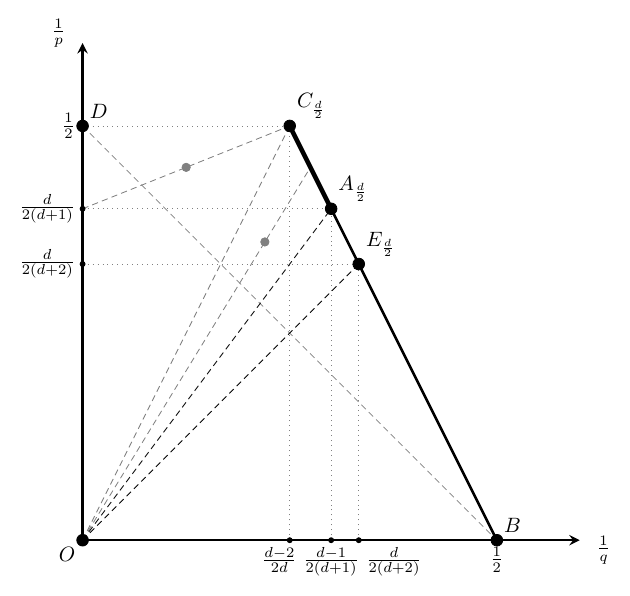}
\caption{$\frac{d}{2}$-admissible pairs. $A$ is the critical point with $OA$ as the critical line, dividing the full admissible region $OBCD$ into the \emph{subcritical} region $OAB$ and the \emph{supercritical} region $OACD$. When $d=2$, point $C$ coalesces with $D$; when $d=1$, points $A$ and $C$ coalesce with $D$.}	
\label{fig1}
\end{figure}

\begin{rem}
The discrete system exponent $\beta$ is sharp on any compact manifold when $\alpha\in(0,1)$ and in the subcritical regime when $\alpha>1$, while it is sharp 
on the sphere in the supercritical regime when $\alpha>1$.  The last case is geometry-sensitive: 
for example,  an improvement is possible on the flat torus, which is however still far from the conjectured sharp exponent $\beta < \frac{d+1}{d}$;  see Corollary 2 and Section 4 of \cite{wang2025strichartz}.
\end{rem}

\begin{rem}
For the wave and Klein–Gordon equations, \cite{wang2025strichartz} proved analogous estimates. In that setting, the result is sharp in the supercritical regime, while it remains conjectured to be sharp in the subcritical regime. See \cite[Theorem 3]{wang2025strichartz} and the remark thereafter.
\end{rem}

\begin{rem}
Unlike the Euclidean case, it is difficult to further remove the $\epsilon$-loss in the frequency-global regularity (when $s$ attains the endpoint $\gamma_{\alpha}(p,q)$ or $\gamma(p,q)$ in Theorem \ref{thm:sharp}, or its counterpart for the wave/Klein-Gordon equations) except at the two endpoints $B$ and $C$. The frequency-localized estimates attain the regularity thresholds. This is mentioned in \cite{wang2025strichartz} and can also be seen from subsequent arguments. However, the situation is different in the non-sharp region (see Theorems~\ref{thm:schodinger} and~\ref{ns2}).
\end{rem}

\subsection{Main new results}
In this paper, we extend Theorem \ref{thm:sharp} 
to non-sharp admissible pairs, 
by combining the estimates on the sharp line due to  \cite{wang2025strichartz} and the extension framework demonstrated originally for Euclidean setting by 
\cite{bez2021strichartz} and \cite{bez2019strichartz}.

\begin{thm}[Fractional Schr\"odinger]\label{thm:schodinger}
    Let $d\ge 1,\,\alpha \in(0,\infty)\setminus\{1\}$, $P=\Delta^{\alpha/2}$ and $I\subset \mathbb R$ be a bounded interval.
Suppose $(p,q)$ is non-sharp $\frac{d}{2}$-admissible in the sense of \eqref{eq:exp-nonsharp} with $\sigma=d/2$. \\
{\rm(I)}  If $(\frac{1}{q},\frac{1}{p})$ lies in the interior of the triangle $O A_{d/2} B$, then for all $s\ge \gamma_{\alpha}(p,q)$ (where $\gamma_{\alpha}(p,q)$ is given by \eqref{eq:gama}) and $\beta\le \beta_{\frac{d}{2}}(p,q)$ (where $\beta_{\frac{d}{2}}(p,q)$ is given by \eqref{eq:beta}), the Strichartz estimate \eqref{eq:str} holds for all orthonormal systems $(f_j)_j$ in $H^s(M)$ and all $\nu=(\nu_j)_j\in\ell^{\beta}(\mathbb{C})$.  \\
{\rm(II)}  If $d\ge 2$ and $(\frac{1}{q},\frac{1}{p})$ lies in the interior of the quadrilateral $O A_{d/2} C_{d/2} D$, then the Strichartz estimate \eqref{eq:str} holds for all $s\ge \gamma_{\alpha}(p,q)$ and $\beta<\frac{p}{2}$. When $d=2$ the quadrilateral reduces to a triangle. 
\end{thm}

\begin{rem}
The Strichartz estimate \eqref{eq:str} also holds on the segment $(O,A_{d/2})$ for all $s\ge \gamma_{\alpha}(p,q)$ and $\beta<\frac{p}{2}$, on $[A_{d/2}, C_{d/2})$ for all $s> \gamma_{\alpha}(p,q)$ and $\beta<\frac{p}{2}$, on $(A_{d/2}, B)$  for $s> \gamma_{\alpha}(p,q)$ and $\beta\le \beta_{d/2}(p,q)$,
on $[C_{d/2},D)$ for $s\ge \gamma_{\alpha}(p,q)$ and $\beta\le \frac{p}{2}=1$, and on $B$ for $s\ge \gamma_{\alpha}(p,q)=0$ and $\beta\le \beta_{d/2}(p,q)=1$. All of these, together with Theorems \ref{thm:schodinger} and \ref{thm:sharp}, solve the Strichartz estimates for the \emph{full} admissible region in \eqref{eq:exp-fulladm} with $\sigma=d/2$ (the origin $O$ and the segment $(O,B)$ are considered only in some localized or weak estimates and are usually not part of our consideration in the final global estimate).
\end{rem}

\begin{rem}
When $\alpha>1$, the range of $\beta$ is sharp in the subcritical triangle $O A_{d/2} B$  on any compact manifold,
and sharp in the supercritical quadrilateral $O A_{d/2} C_{d/2} D$ on the sphere $\mathbb{S}^d$. 
As discussed in Section \ref{sec:impr},  the range of $\beta$ in $O A_{d/2} C_{d/2} D$ can be improved on the flat torus.  See Corollary \ref{cor:OACD}. Subcritical regime is already sharp on any manifold from the Weyl law, but Corollary \ref{cor0}, on the flat torus, lies in the more relaxed regularity threshold $\gamma$ and the corresponding $\beta$ increases with $\gamma$.
When $\alpha\in(0,1)$, the sharpness follows by the long-time frequency-localized estimate and the construction of $j$--zonal eigenfunctions.
See Appendix \ref{app:nec} for further details on necessity. 
\end{rem}
\begin{rem}
The case $\beta=p/2$ for the supercritical region $O A_{d/2} C_{d/2} D$ is reduced to a localized estimate on the critical line  $[O,A_{d/2})$. See Proposition \ref{prop:beta=p/2}. However, Example \ref{ex} tells that global estimate cannot hold on the critical line. 
\end{rem}

\begin{thm}[Wave and Klein-Gordon]\label{ns2}
    Let $d\ge 2, m\ge 0$, $P=\sqrt{\Delta+m^2}$ and $I\subset \mathbb R$ be a bounded interval. Suppose $(p,q)$ is non-sharp $\frac{d-1}{2}$-admissible in the sense of \eqref{eq:exp-nonsharp} with $\sigma=(d-1)/2$.\\   
{\rm(I)} If $(\frac{1}{q},\frac{1}{p})$ lies in the interior of the triangle $O A_{(d-1)/2} B$, then
for all $s\ge\gamma(p,q)$ given by \eqref{eq:gam} and $\beta\le\beta_{\frac{d-1}{2}}(p,q)$ given by \eqref{eq:beta},
 the Strichartz estimate \eqref{eq:str}   holds for all orthonormal systems $(f_j)_j$ in $H^s(M)$ and all  $\nu=(\nu_j)_j\in \ell^\beta(\mathbb C)$.\\
{ \rm(II)} If $d\geq 3$ and  $(\frac{1}{q},\frac{1}{p})$ lies in  the interior of $O A_{(d-1)/2} C_{(d-1)/2} D$, 
the Strichartz estimate \eqref{eq:str} holds for all $s \ge \gamma(p,q)$ and $\beta<\frac{p}{2}$.
\end{thm}

\begin{figure}[h] \centering
\includegraphics[width=0.8\textwidth]{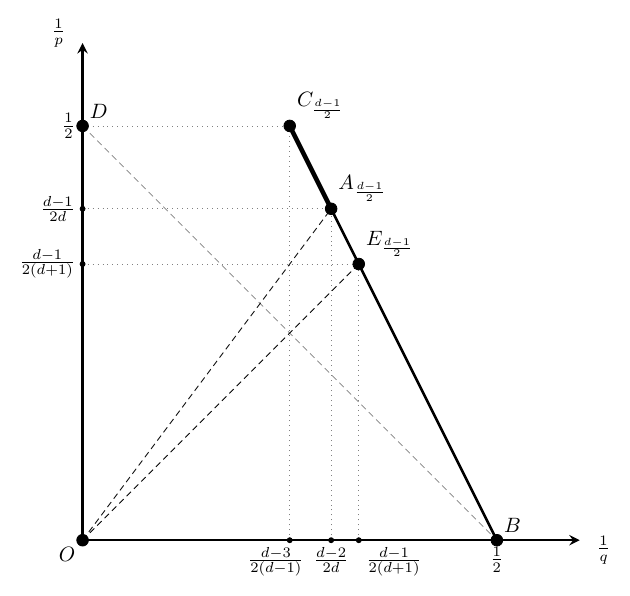}
\caption{$\frac{d-1}{2}$-admissible pairs}
\label{fig2}
\end{figure}

 \begin{rem}
 The range of  $\beta$ in Theorem \ref{ns2} is sharp in the supercritical quadrilateral $O A_{(d-1)/2} C_{(d-1)/2} D$ on any compact manifold;
however, whether $\beta$ is sharp in the subcritical triangle $O A_{(d-1)/2} B$ remains an open question. This differs from the fractional Schr\"odinger case. In fact, the same issue also arises in the Euclidean setting.
\end{rem}

\subsection{Structure of the paper}
Section~\ref{sec:proof1} presents the first proof of the main theorems with $\epsilon$-loss of regularity using the Lieb-Sobolev inequality, while Section~\ref{sec:proof2} provides an alternative proof based on localized weak Lorentz estimates. Section~\ref{sec:impr} is dedicated to the improvement on the flat torus, followed by a final appendix on the sharpness of the system exponents.

\section{Proof of main theorems with $\epsilon$-loss of regularity}\label{sec:proof1} 

\subsection{Cwikel estimate and Schatten class}
We first state a Cwikel estimate on compact manifolds, due to Sukochev-Yang-Zanin \cite[Lemma 2.3]{sukochev2025singular}. Although their result holds for all $q>0$, we only need the case $q>2$. 
Before stating it, we recall the Schatten classes.

\begin{defi}
For $0<p,q<\infty$, the Lorentz norm on a $\sigma$-finite measure space $(X,\mu)$ is defined by
\[
\|f\|_{L^{p,q}} = \left( p \int_0^\infty \lambda^{\,q-1} \bigl(d_f(\lambda)\bigr)^{q/p} \, d\lambda \right)^{1/q},
\] 
where $d_f(\lambda) = \mu\bigl(\{x\in X : |f(x)| > \lambda\}\bigr)$ is the distribution function of $f$.
For $q=\infty$ and $0<p\le\infty$, the weak $L^p$ norm is defined by
\[
\|f\|_{L^{p,\infty}} = \sup_{\lambda>0} \; \lambda \; \bigl(d_f(\lambda)\bigr)^{1/p}.
\]
The Schatten class $\mathfrak S^{p,q}(H)$ consists of all compact operators $T$ on a Hilbert space $H$ such that the singular values satisfy $\lambda(T)\in\ell^{p,q}$, with Schatten norm
\[
\|T\|_{\mathfrak S^{p,q}} = \|\lambda(T)\|_{\ell^{p,q}}.
\] 
In particular, when $p=q$, we denote the corresponding class by $\mathfrak S^{p}(H)$.
\end{defi}

\begin{lem}[Cwikel estimate]\label{lem0}
    Let $d\ge1$ and $q>2$. 
    Then for all $W\in L^q(M)$,
    \[
        \bigl\| W\,(1+\Delta)^{-\frac{d}{2q}} \bigr\|_{\mathfrak{S}^{q,\infty}(L^2(M))} \lesssim \|W\|_{L^q(M)} .
    \]
\end{lem}
Here $W$  is interpreted as a multiplication operator on the left-hand side, as are all functions treated as operators throughout this paper.

\subsection{Lieb-Sobolev inequality}
We next establish a Sobolev inequality for orthogonal systems on compact manifolds, which was used by Bez-Lee-Nakamura \cite{bez2021strichartz} in the Euclidean setting. The underlying ideas originate from the work of Lieb \cite{lieb1983lp}.

\begin{lem}[Lieb-Sobolev inequality]\label{lemma}
    Let $d\ge 1$ and $1<p<\infty$. Then the inequality
    \begin{equation}\label{eq:LS}
        \Big\| \sum_{j} \nu_j |\langle D\rangle^{-\frac{d}{2p'}} f_j|^2\Big\|_{L^{p}(M)} \lesssim \| \nu \|_{\ell^{p,1}}
    \end{equation}
    holds for all orthonormal systems $(f_j)_j$ in $L^2(M)$ and  $\nu=(\nu_j)_j$ in $\ell^{p,1}(\mathbb C)$.
\end{lem}

\begin{proof}
    Let $s = \frac{d}{2p'}$ and define
\[
g = \sum_j \nu_j \bigl| (1+\Delta)^{-s/2} f_j \bigr|^2.
\]
We aim to estimate $\|g\|_{L^p(M)}$.  By duality, it suffices to estimate the integral
\begin{equation}\label{eq:lem2-int}
\left|\int_M g(x) V(x) \,dx\right|\lesssim \| \nu \|_{\ell^{p,1}} 
\end{equation}
for every $V \in L^{p'}(M)$ with $\|V\|_{L^{p'}(M)} = 1$.

   Let $J=(1+\Delta)^{-\frac{1}{2}}$ be the Bessel potential operator, and define
$$K=V^{\frac{1}{2}}J^s,\qquad K^*=J^sV^{\frac{1}{2}}.$$
    Here $K^*$ is the adjoint of $K$. Let $L=K^* K$. Define the diagonal operator
    \[\Gamma = \sum_{j} \nu_j |f_j \rangle \langle f_j|,\] where we use Dirac’s notation $|f\rangle \langle g|(h):=\langle g, h \rangle f$. 
 By the definitions, we have
    \begin{equation}\label{eq:lem2-gam}
        \|\Gamma\|_{\mathfrak S^{p,1}} = \|\nu\|_{\ell^{p,1}}.
    \end{equation}
Then we have the l.h.s. of \eqref{eq:lem2-int}
\begin{align*}
\int_M g(x) V(x) dx 
&=\sum_{j} \nu_j \int_M  |V^{\frac{1}{2}}(x)(1+\Delta)^{-\frac{s}{2}} f_j(x)|^2  dx \\
&= \sum_{j} \nu_j \langle K f_j, K f_j \rangle_{L^2(M)} \\
&= \sum_{j} \nu_j \langle f_j, L f_j \rangle_{L^2(M)} \\
&= \operatorname{Tr}(\Gamma L).
\end{align*}

Since $V \in L^{p'}(M)$, it follows that $V^\frac{1}{2} \in L^{2p'}(M)$. Note that $p >1$ implies $2p' > 2$.
    By Lemma \ref{lem0}, with $q=2p'$ and $W=V^\frac{1}{2} $, we have 
\[
        \|K\|_{\mathfrak S^{q,\infty}} =\Big\| V^\frac{1}{2} (1+\Delta)^{-\frac{s}{2}}\Big\|_{\mathfrak S^{q,\infty}} \lesssim \|V^\frac{1}{2}\|_{L^q(M)} = \|V\|_{L^{p'}(M)}^\frac{1}{2}=1.
\]
    By definition, the singular values satisfy
    $$\lambda_n(L)=\lambda_n(K^* K)=\lambda_n(K)^2,$$
    where $\lambda_n$ denotes the $n$-th singular value in decreasing order.
    Thus, we have
    \begin{equation}\label{eq:lem2-L}
        \|L\|_{\mathfrak S^{p',\infty}} = \|K\|_{\mathfrak S^{q,\infty}}^2 \lesssim 1.
    \end{equation}
By \eqref{eq:lem2-gam}, \eqref{eq:lem2-L} and H\"older inequality for Schatten norms, we finally obtain
    \[
       | \text{Tr}(\Gamma L) | \leq \|\Gamma\|_{\mathfrak S^{p,1}} \|L\|_{\mathfrak S^{p',\infty}} \lesssim \|\nu\|_{\ell^{p,1}} .
   \]
This will give \eqref{eq:lem2-int} and \eqref{eq:LS}.
\end{proof}

\subsection{Reduction principle}
We now explain how to obtain non-sharp region Strichartz estimates from sharp line estimates and the Lieb--Sobolev estimate.
\begin{prop}[Sharp line to non-sharp region]\label{prop1}
Let $d\geq 1$, $\sigma \geq \frac{1}{2}$, $a\in \mathbb{R}$, and define
\begin{equation}\label{eq:sapq}
s_a(p,q)=\frac{a}{p}+d\,\Bigl(\frac{1}{2}-\frac{1}{q}\Bigr).\end{equation}
Assume that, for all sharp $\sigma$-admissible pairs $(p,q)$ with $\bigl(\frac{1}{q},\frac{1}{p}\bigr) \in (A_{\sigma}, E_{\sigma})$, the estimate
\begin{equation}\label{ns3}
\Bigl\| \sum_{j} \nu_j \bigl|e^{itP}\langle D\rangle^{-s_a(p,q)} f_j\bigr|^2\Bigr\|_{L_t^{p/2}L_x^{q/2}(I\times M)} \lesssim  \| \nu \|_{\ell^\beta}
\end{equation}
holds with $\beta=\beta_\sigma(p,q)$ given by \eqref{eq:beta}, for all orthonormal systems $(f_j)_j$ in $L^2(M)$ and all $\nu=(\nu_j)_j \in \ell^{\beta}(\mathbb C)$.
Then \eqref{ns3} also holds with $\beta=\beta_\sigma(p,q)$ for all non-sharp $\sigma$-admissible pairs $(p,q)$ satisfying $p>\frac{2\sigma-1}{2\sigma}q$, i.e., $\bigl(\frac{1}{q},\frac{1}{p}\bigr)$ lies in the interior of triangle $OA_\sigma B$.
\end{prop}

\begin{proof}
The segment $(A_{\sigma}, E_{\sigma})$ corresponds to the range $\frac{2(\sigma+1)}{\sigma} <q<\frac{2(2\sigma+1)}{2\sigma-1}$.
For simplicity, assume $\sigma= \frac{d}{2}$ and $a=-1$ so that $s_a(p,q)=\gamma_\alpha(p,q)$ as the argument is independent of these parameters. See \eqref{eq:AE}, \eqref{eq:gama} and Figure \ref{fig1}.  
    
From \eqref{eq:beta} and \eqref{eq:sapq}, we have
\[
s_a(\infty,q) = d\left(\frac{1}{2}-\frac{1}{q}\right) = \frac{d}{2} \cdot \frac{1}{(q/2)'}, \qquad \beta_\sigma(\infty,q) = \frac{q}{2}.
\] Since the dispersive flow $(e^{itP}f_j)_j$ preserves orthonormality for any orthonormal initial data $(f_j)_j \subset L^2(M)$, Lemma \ref{lemma}~ yields for any $q>2$
\begin{equation}\label{ns4}
\Big\| \sum_{j} \nu_j |e^{itP}\langle D\rangle^{-s_a(\infty, q)} f_j|^2\Big\|_{L^\infty_t L^{q/2}_x(I\times M)} \lesssim \| \nu \|_{\ell^{\frac{q}{2},1}}.
\end{equation}
Applying complex interpolation between \eqref{ns3} and \eqref{ns4}, we obtain the weaker (Lorentz) Strichartz estimate
\[
\Big\| \sum_{j} \nu_j |e^{itP}\langle D\rangle^{-s_a(p,q)} f_j|^2\Big\|_{L_t^{p/2}L_x^{q/2}(I\times M)} \lesssim \| \nu \|_{\ell^{\beta_\sigma(p,q),1}},
\]
    valid for $(\frac{1}{q},\frac{1}{p})$ in the interior of triangle $OA_{d/2}B$.  The interpolation can proceed directly or in two stages: first between the segment $(A_{d/2}, E_{d/2})$ and point $B$ to obtain estimates on $(A_{d/2}, B]$, then between segment $(A_{d/2}, B]$ and points near $O$ to reach the interior. See Figure \ref{fig1}. 
    
Although the discrete exponent $\beta$ is already sharp in the weaker Lorentz form, we use real interpolation to upgrade it to strong-type estimates. It suffices to work in the triangle $OA_{d/2}E_{d/2}$, since another interpolation extends the result to the full region.

 Fix $(\frac{1}{q},\frac{1}{p})$ in the interior of $OA_{d/2}E_{d/2}$ and choose $(\frac{1}{q_i},\frac{1}{p_i})$ ($i=0,1$) in the same region with $s_a(p,q) = s_a(p_0,q_0) = s_a(p_1,q_1)$. Then for some $\theta \in (0,1)$,
\[
\frac{1}{p} = \frac{1-\theta}{p_0} + \frac{\theta}{p_1}, \quad \frac{1}{q} = \frac{1-\theta}{q_0} + \frac{\theta}{q_1}.
\]     
 By real interpolation (see for example \cite[(3.5-3.6)]{bez2019strichartz} or references therein),
\[(L^{\frac{p_0}{2}}_t L^{\frac{q_0}{2}}_x, L^{\frac{p_1}{2}}_t L^{\frac{q_1}{2}}_x)_{\theta,\frac{p}{2}} = L^{\frac{p}{2}}_t L^{\frac{q}{2},\frac{p}{2}}_x, \qquad 
(\ell^{\beta_\sigma(p_0,q_0),1}, \ell^{\beta_\sigma(p_1,q_1),1})_{\theta,\frac{p}{2}} = \ell^{\beta_\sigma(p,q),\frac{p}{2}},\]
and consequently,
\[
\Big\| \sum_{j} \nu_j |e^{itP}\langle D\rangle^{-s_a(p,q)} f_j|^2\Big\|_{L^{\frac{p}{2}}_t L^{\frac{q}{2},\frac{p}{2}}_x(I\times M)} \lesssim \| \nu \|_{\ell^{\beta_\sigma(p,q),p/2}}.
\]
Since $p<q$ and $ \beta_\sigma(p,q)<p/2$, we have the embeddings
    $$L^{\frac{p}{2}}_t L^{\frac{q}{2},\frac{p}{2}}_x \subset L^{\frac{p}{2}}_t L^{\frac{q}{2}}_x, \qquad \ell^{\beta_\sigma(p,q)}\subset \ell^{\beta_\sigma(p,q),\frac{p}{2}},$$
and consequently the strong estimate
    \begin{equation}\label{ns5}
        \Big\| \sum_{j} \nu_j |e^{itP}\langle  D\rangle^{-s_a(p,q)} f_j|^2\Big\|_{L_t^{p/2}L_x^{q/2}(I\times M)} \lesssim  \| \nu \|_{\ell^{\beta_\sigma (p,q)}}
    \end{equation}
for all $(\frac{1}{q},\frac{1}{p})$ in the interior of the triangle $O A_{d/2} E_{d/2}$.

Finally, by Plancherel's theorem,
\[
\|e^{itP} f\|_{L^\infty_t L^2_x} \lesssim \|f\|_{L^2},
\]
which gives \eqref{ns5} at the trivial point $B$ (where $(p,q,s,\beta) = (\infty,2,0,1)$) via the triangle inequality. 

A final interpolation extends \eqref{ns5} to the interior of triangle $OA_{d/2}B$.
\end{proof}

\subsection{Proof of Theorem \ref{thm:schodinger}}
(I) Let $(p,q)$ be a non-sharp $\frac{d}{2}$-admissible pair lying in the interior of the triangle $O A_{d/2} B$. Now we want to prove the Strichartz estimate \eqref{eq:str} with $P=\Delta^{\alpha/2}$ holds with $\beta=\beta_{d/2}(p,q)$ for all orthonormal systems $(f_j)_j\in H^s(M)$ with $s>\gamma_\alpha(p,q)$.

For every fixed $s>\gamma_{\alpha}(p,q)$, there exists a constant $a \in \mathbb{R}$ such that $s=s_a(p,q)=\frac{a}{p}+d(\frac{1}{2}-\frac{1}{q})$.
For every fixed orthonormal system $(f_j)_j$ in $H^{s}(M)$, let $g_j=\langle D \rangle^{s} f_j$, 
then $(g_j)_j$ is an orthonormal system in $L^2(M)$.

By Theorem \ref{thm:sharp},  if $(p,q)$ is sharp $\frac{d}{2}$-admissible   such that $(\frac{1}{q},\frac{1}{p}) \in (A_{d/2},B]$,  the inequality
\[
    \Big\|\sum_j  \nu_j |e^{it\Delta^{\alpha/2}} f_j|^2\Big\|_{L_t^{p/2}L_x^{q/2}(I\times M)} \lesssim 
    \|\nu\|_{\ell^\beta}
\]
holds with $\beta=\beta_{\frac{d}{2}}(p,q)=\frac{2q}{q+2}$.
Hence on the line segment $(A_{d/2}, B]$, we have
\[
    \Big\|\sum_j  \nu_j |e^{it\Delta^{\alpha/2}} \langle D \rangle^{-s_a(p,q)} g_j|^2\Big\|_{L_t^{p/2}L_x^{q/2}(I\times M)} \lesssim 
    \|\nu\|_{\ell^\beta}
\]
with the same system exponent $\beta=\beta_{\frac{d}{2}}(p,q)$.
By Proposition \ref{prop1}, we obtain that,  for all non-sharp admissible exponents $(p,q)$ with $(1/q,1/p)$ lying in the interior of the triangle $OA_{d/2}B$,  the inequality
\[
  \Big\|\sum_j  \nu_j |e^{it\Delta^{\alpha/2}} f_j|^2\Big\|_{L_t^{p/2}L_x^{q/2}(I\times M)} =   \Big\|\sum_j  \nu_j |e^{it\Delta^{\alpha/2}} \langle D \rangle^{-s_a(p,q)} g_j|^2\Big\|_{L_t^{p/2}L_x^{q/2}(I\times M)} \lesssim 
    \|\nu\|_{\ell^\beta}
\]
holds  with $\beta=\beta_{\frac{d}{2}}(p,q)$. Since $(f_j)_j$ was arbitrary, this proves the theorem in the triangle $OA_{d/2}B$.

(II) Now we consider the case when $(\frac{1}{q},\frac{1}{p})$ lies in the interior of the quadrilateral $O A_{d/2} C_{d/2} D$.
We first prove that  \eqref{eq:str} holds with $\beta =1$, on the line segment $[C_{d/2}, D)$.
For all $(1/q,1/p)\in (C_{d/2},D)$,
by Theorem \ref{sig0} and Minkowski's inequality, we have,  
for all orthonormal systems $(f_j)_j$ in $H^{s}(M)$ with $s\geq\gamma_{\alpha}(p,q)$, 
\begin{align*}
    \Big\|\sum_j  \nu_j |e^{it\Delta^{\alpha/2}} f_j|^2\Big\|_{L_t^{p/2}L_x^{q/2}(I\times M)} &\leq \sum_j |\nu_j| \|e^{it\Delta^{\alpha/2}} f_j\|_{L_t^{p}L_x^{q}(I\times M)}^2 \\
    &\lesssim \sum_j |\nu_j| \|f_j\|_{H^s(M)}^2 \\
    &= \|\nu\|_{\ell^1}.
\end{align*}
Note that, when $(\frac{1}{q},\frac{1}{p})$ lies on the line segment $(O, A_{d/2})$, $\beta_{\frac{d}{2}}(p,q)=\frac{p}{2}$.
Applying the interpolation between the estimates \eqref{eq:str} for points on the line segment $(C_{d/2}, D)$ 
and the estimates for points inside the triangle $O A_{d/2} B$ that approach the line segment $(O, A_{d/2})$ arbitrarily closely, 
we complete the proof in the interior of the quadrilateral $O A_{d/2} C_{d/2} D$.

\subsection{Proof of Theorem \ref{ns2}}
Proof of the Wave/Klein-Gordon case is similar. 

(I) Let $(p,q)$ be a non-sharp $\frac{d-1}{2}$-admissible pair lying in the interior of the triangle $O A_{\frac{d-1}2} B$. Now we want to prove the Strichartz estimate \eqref{eq:str} with $P=\sqrt{m^2+\Delta}$ holds with $\beta=\beta_{\frac{d-1}2}(p,q)$ for all orthonormal systems $(f_j)_j\in H^s(M)$ with $s>\gamma(p,q)$.

For every fixed $s>\gamma(p,q)$, there exists a constant $a \in \mathbb{R}$ such that $s=s_a(p,q)=\frac{a}{p}+d(\frac{1}{2}-\frac{1}{q})$.
For every fixed orthonormal system $(f_j)_j$ in $H^{s}(M)$, let $g_j=\langle D \rangle^{s} f_j$, 
then $(g_j)_j$ is an orthonormal system in $L^2(M)$.

By the wave/Klein-Gordon  version of Theorem \ref{thm:sharp} (\cite[Theorem 3]{wang2025strichartz}),  if $(p,q)$ is sharp $\frac{d-1}{2}$-admissible   such that $(\frac{1}{q},\frac{1}{p}) \in (A_{\frac{d-1}2},B]$,  the inequality
\[
    \Big\|\sum_j  \nu_j |e^{it\sqrt{m^2+\Delta}} f_j|^2\Big\|_{L_t^{p/2}L_x^{q/2}(I\times M)} \lesssim 
    \|\nu\|_{\ell^\beta}
\]
holds with $\beta=\beta_{\frac{d-1}{2}}(p,q)=\frac{2q}{q+2}$.
Hence on the line segment $(A_{\frac{d-1}2}, B]$, we have
\[
    \Big\|\sum_j  \nu_j |e^{it\sqrt{m^2+\Delta}} \langle D \rangle^{-s_a(p,q)} g_j|^2\Big\|_{L_t^{p/2}L_x^{q/2}(I\times M)} \lesssim 
    \|\nu\|_{\ell^\beta}
\]
with the same system exponent $\beta=\beta_{\frac{d-1}{2}}(p,q)$.
By Proposition \ref{prop1}, we obtain that,  for all non-sharp admissible exponents $(p,q)$ with $(1/q,1/p)$ lying in the interior of the triangle $OA_{\frac{d-1}2}B$, the inequality
\begin{multline*}
  \Big\|\sum_j  \nu_j |e^{it\sqrt{m^2+\Delta}} f_j|^2\Big\|_{L_t^{p/2}L_x^{q/2}(I\times M)} =  \Big\|\sum_j  \nu_j |e^{it\sqrt{m^2+\Delta}} \langle D \rangle^{-s_a(p,q)} g_j|^2\Big\|_{L_t^{p/2}L_x^{q/2}(I\times M)}\\
   \lesssim 
    \|\nu\|_{\ell^\beta}
\end{multline*}
holds  with $\beta=\beta_{\frac{d-1}{2}}(p,q)$. Since $(f_j)_j$ was arbitrary, this proves the theorem in the triangle $OA_{\frac{d-1}2}B$.

(II) Now we consider the case when $(\frac{1}{q},\frac{1}{p})$ lies in the interior of $O A_{(d-1)/2} C_{(d-1)/2} D$.
We first prove via Theorem \ref{sig1} and Minkowski's inequality that, on the line segment $(C_{(d-1)/2}, D)$, the inequality
\[
    \Big\|\sum_j  \nu_j |e^{it\sqrt{\Delta+m^2}} f_j|^2\Big\|_{L_t^{p/2}L_x^{q/2}(I\times M)} \leq   \|\nu\|_{\ell^1}
\]
holds for all orthonormal systems $(f_j)_j$ in $H^{s}(M)$ with $s\geq\gamma(p,q)$.
Note that, when $(\frac{1}{q},\frac{1}{p})$ lies on the line segment $(O, A_{(d-1)/2})$,  $\beta_{\frac{d-1}{2}}(p,q)=\frac{p}{2}$.
By interpolation  between  points on the line segment $(C_{(d-1)/2},D)$ 
and points inside the triangle $O A_{(d-1)/2} B$ that approach the line segment $(O, A_{(d-1)/2})$ arbitrarily closely, 
we complete the proof in the interior of $O A_{(d-1)/2} C_{(d-1)/2} D$.

\section{Another proof based on localized weak Lorentz estimates} \label{sec:proof2}

The proof in Section \ref{sec:proof1} is based on Proposition \ref{prop1}, following the strategy of \cite{bez2021strichartz}.
In this section, we present another proof based on \cite{bez2019strichartz}. 
This second approach does not use the Lieb--Sobolev inequality from Lemma~\ref{lemma}; instead, it combines frequency-localized estimates with a globalization argument in weak Lorentz spaces.

We recall the duality principle of Frank--Sabin \cite{frank2017restriction} and two frequency-localized dispersive estimates, 
due respectively to \cite{MR2058384, dinh2016strichartz} and \cite{cacciafesta2024strichartz, fio}.

\begin{lem}[Duality principle]\label{dua} 
    Let $p,q\ge 2$, $r_1, r_2, \alpha, \beta \ge 1$. Suppose that $T$ is a bounded operator from $L^2$ to $L_t^{p,2r_1} L_x^{q,2r_2}$. Then the inequality
    \[
    \Bigl\|\sum_j \nu_j |T f_j|^2\Bigr\|_{L_t^{\frac{p}{2}, r_1} L_x^{\frac{q}{2}, r_2}} \le C \|\nu\|_{\ell^{\alpha,\beta}}
    \]
    holds for all orthonormal systems $(f_j)_j$ in $L^2$ and all $\nu = (\nu_j)_j$ in $\ell^{\alpha,\beta}$ if and only if the inequality
    \[
    \|W T T^* \overline{W}\|_{\mathfrak S^{\alpha',\beta'}} \le C \|W\|_{L_t^{2(\frac{p}{2})', 2r_1'} L_x^{2(\frac{q}{2})', 2r_2'}}^2
    \]
    holds for all $W \in L_t^{2(\frac{p}{2})', 2r_1'} L_x^{2(\frac{q}{2})', 2r_2'}$.
\end{lem}

\begin{lem}\label{din}
	Let $\alpha\in (0,\infty)\setminus\{1\}$ and $\varphi\in C_0^\infty(\mathbb{R}\setminus\{0\})$. Then there exist $t_0>0$ and $C>0$ such that for all $h\in (0,1]$ and $t\in [-t_0h^{\alpha-1},t_0h^{\alpha-1}]$,  
	\[
	\|e^{it\Delta^{\alpha/2}}\varphi(h\sqrt{\Delta})f\|_{L^\infty(M)}\le Ch^{-d}(1+|t|h^{-\alpha})^{-d/2}\|f\|_{L^1(M)}.
	\]

\end{lem}

\begin{lem}\label{dislem}
	Let $\varphi\in C_0^\infty(\mathbb{R}\setminus\{0\})$. There exist $t_0>0$ and $C>0$ such that for all $h\in (0,1]$ and $t\in [-t_0,t_0]$,
	\[
	\|e^{it\sqrt{m^2+\Delta}}\varphi(h\sqrt{\Delta})f\|_{L^\infty(M)}\le Ch^{-d}(1+|t|/h)^{-(d-1)/2}\|f\|_{L^1(M)}.
	\]
\end{lem}

\subsection{Fractional Schr\"odinger case}

The key tool is the following frequency-localized version of Theorem \ref{thm:schodinger}.
\begin{prop}\label{prop:local}
    Let $d \geq 1$, $\alpha \in (0,\infty)\setminus\{1\}$. Let $P_l$ be the frequency-localized operator given by $P_l = \varphi(2^{-l} \sqrt{\Delta})$ for $l\in\mathbb Z$ and $\varphi \in C_0^\infty(\mathbb{R}\setminus\{0\})$. 
Then for all $\frac{d}{2}$-admissible pairs $(p,q)$ such that $(\frac{1}{q},\frac{1}{p})$ lies in the interior of the triangle $O A_{d/2} E_{d/2}$, and for all $l\in\mathbb Z$, we have
    \begin{equation}\label{local2}
        \Big\|\sum_j  \nu_j |e^{it\Delta^{\alpha/2}}P_l f_j|^2\Big\|_{L_t^{p/2}L_x^{q/2}(I\times M)} \lesssim
        2^{2l\gamma_{\alpha}(p,q)}\|\nu\|_{\ell^\beta}
    \end{equation}
for  all orthonormal systems $(f_j)_j$ in $L^2(M)$ and all sequences $\nu=(\nu_j)_j\in \ell^\beta$ with $\beta=\beta_{d/2}(p,q)$.
\end{prop}

\begin{proof}
 
We first claim that the following endpoint estimate (at point $O$) holds for all orthonormal systems $(f_j)_j$ in $L^2(M)$:
    \begin{equation}\label{local3}
        \Big\|\sum_j  \nu_j |e^{it\Delta^{\alpha/2}}P_l f_j|^2\Big\|_{L_t^{\infty}L_x^{\infty}(I\times M)} \lesssim
        2^{ld}\|\nu\|_{\ell^\infty}.
    \end{equation}
To see this, let $(e_k)_k$ be an orthonormal eigenbasis in $L^2(M)$ associated with the eigenvalues $(\la_k)_k$ of $\sqrt{\Delta}$, where $0=\la_0 \leq \la_1 \leq \cdots$.
    Let $K_l(t,x,y)$ be the kernel of the operator $e^{it\Delta^{\alpha/2}}P_l$, namely
    \[
        K_l(t,x,y) = \sum_k e^{it\la_k^\alpha} \varphi(2^{-l}\la_k) e_k(x)\overline{e_k(y)}.
    \]
    Then by Bessel's inequality and the local Weyl estimate on compact manifolds, we have
  \[        \sum_j |e^{it\Delta^{\alpha/2}}P_l f_j(x)|^2 = \sum_j \Big|\int_M K_l(t,x,y) f_j(y) dy\Big|^2 \leq \|K_l(t,x,\cdot)\|_{L^2(M)}^2 \lesssim 2^{ld} .\]

    Based on \eqref{local3} and complex interpolation,  it remains only to prove \eqref{local2} on the interval $(A_{d/2}, E_{d/2}]$.
    Since the localization operator $P_l$ need not preserve orthogonality, the frequency-localized version of Theorem \ref{thm:sharp} cannot be applied directly.
    However, the Frank-Sabin duality principle remains valid in the presence of the localization operator $P_l$.
    Hence, we can obtain the estimate \eqref{local2} on the line segment $(A_{d/2}, E_{d/2}]$ using an argument similar to the proof of Theorem \ref{thm:sharp}.
When $\alpha>1$, we need to split the interval $I$ into subintervals $\left\{ I_{l,n}\right\}$ of length $2^{(1-\alpha)l}$. Then  
    \begin{align*}
        \Big\|\sum_j  \nu_j |e^{it\Delta^{\alpha/2}}P_l f_j|^2\Big\|_{L_t^{p/2}L_x^{q/2}(I\times M)}^{p/2} 
        &\leq \sum_n \int_{I_{l,n}} \Big\|\sum_j  \nu_j |e^{it\Delta^{\alpha/2}}P_l f_j|^2\Big\|_{L_x^{q/2}}^{p/2} dt \\
        &\lesssim 2^{l(\alpha-1)} \max_n \Big\|\sum_j  \nu_j |e^{it\Delta^{\alpha/2}}P_l f_j|^2\Big\|_{L_t^{p/2}L_x^{q/2}(I_{l,n}\times M)}^{p/2}.
    \end{align*}
    
    For any small $\eps>0$, we introduce a holomorphic family of operators $T^\eps_{z,l}$ in the strip \(\{z\in\mathbb{C}: -r/2\le \text{Re}z\le 0\}\), 
whose kernel is given by
    \[K^\eps_{z,l}(t,x,s,y)=\1_{|t|\le 2^{(1-\alpha)l}}\1_{|s|\le 2^{(1-\alpha)l}}\1_{\eps< |t-s|}(t-s)^{-1-z}\sum_{k} \varphi(2^{-l}\lambda_k)^2 e^{i (t-s)\la_k^\alpha}e_k(x)\overline{e_k(y)}.  \]
    By Lemma \ref{dua}, it suffices to estimate $\|WT^\eps_{-1,l}  \overline{W}\|_{\mathfrak S^{r}}$, where $r=\beta'$.
    We shall obtain this estimate by applying Stein's complex interpolation theorem to the above analytic family of operators. We consider the two boundaries of the strip.
Let $z_1=-\frac r2+ib$ with $b\in\mathbb{R}$. 
    Using the Hardy–Littlewood–Sobolev inequality together with the frequency-localized dispersive estimate from Lemma \ref{din}, we have
    \begin{align*}
        \|WT^\eps_{z_1,l}  \overline{W}\|_{\mathfrak S^{2}}\lesssim 2^{\frac{(2-\alpha)dl}{2}} \|W\|_{L_t^{\frac4{r-d}}L_x^2}^2.
    \end{align*} 
    The use of the Hardy--Littlewood--Sobolev inequality imposes the condition $r \in (d+1,d+2)$, which corresponds to the segment $(A_{d/2}, E_{d/2})$.
   Let $z_2=ib$ with $b\in\mathbb{R}$. 
    In this case, by the Plancherel theorem and the uniform boundedness of the truncated Hilbert transform, we have
    \begin{equation}\label{eq:infinitybound}
        \|WT^\eps_{z_2,l}  \overline{W}\|_{\mathfrak S^{\infty}}\lesssim  (1+|b|) \|W\|_{L^\infty_{t}L^\infty_{x}}^2.
    \end{equation}
    Interpolating between these two boundary estimates yields
    \begin{align*}
        \|WT^\eps_{-1,l} \overline{W}\|_{\mathfrak S^{r}}\lesssim   2^{\frac2p(2-\alpha)l} \|W\|_{L_t^{\frac{2r}{r-d}}L_x^{r}(M)}^2
    \end{align*}
    for all $W\in L_t^{\frac{2r}{r-d}}L_x^{r}(I\times M)$.
    Consequently, letting $\eps \to 0$, Lemma \ref{dua} implies
    \begin{equation}\label{local4}
        \Big\|\sum_j  \nu_j |e^{it\Delta^{\alpha/2}}P_l f_j|^2\Big\|_{L_t^{p/2}L_x^{q/2}(I_{l,n}\times M)} \lesssim 2^{\frac{2}{p}(2-\alpha)l} \|\nu\|_{\ell^\beta}.
    \end{equation}
    After summing over the short time intervals, we further get 
    \begin{equation}\label{local5}
        \Big\|\sum_j  \nu_j |e^{it\Delta^{\alpha/2}}P_l f_j|^2\Big\|_{L_t^{p/2}L_x^{q/2}(I\times M)} \lesssim 2^{2l/p}\|\nu\|_{\ell^\beta}.
    \end{equation}
    Finally, when $\alpha>1$, the desired estimate \eqref{local2} follows by complex interpolation between \eqref{local3} and \eqref{local5}. 
   When $0<\alpha<1$, no decomposition of the time interval is needed, and the same conclusion follows directly by interpolating between \eqref{local3} and \eqref{local4}.
\end{proof}

\begin{rem}
Interpolating with the trivial point $B$, Proposition \ref{prop:local} holds for the triangle $OA_{d/2}B$ excluding the side $(O,A_{d/2}]$.
\end{rem}

We now state the globalization lemma, which is due to \cite[Proposition 2.2]{bez2019strichartz}.
\begin{lem}\label{lem2}
    Let $p_0,p_1 > 2$, $q \geq 2$, $\beta_0,\beta_1 \geq 1$, and let $(g_j)_j$ be a sequence of functions such that for each $i=0,1$, it is uniformly bounded in $L^{p_i}_t L^{q}_x(I\times M)$. 
    Assume that for each $i$, there exists $\varepsilon_i >0$ such that for all $l \in \mathbb{Z}$ and $\nu\in \ell^{\beta_i}$,
    \[
        \Big\|\sum_j  \nu_j |P_l g_j|^2\Big\|_{L_t^{\frac{p_i}{2},\infty}L_x^{\frac{q}{2}}(I\times M)} \lesssim 
        2^{(-1)^{i+1}\varepsilon_i l}\|\nu\|_{\ell^{\beta_i}}
    \]
    where $P_l$ is the Littlewood-Paley projection onto frequencies of scale $2^l$.
    Then for all $\nu = (\nu_j)_j \in \ell^{\beta,1}$ we have 
    \[
        \Big\|\sum_j  \nu_j | g_j|^2\Big\|_{L_t^{\frac{p}{2},\infty}L_x^{\frac{q}{2}}(I\times M)} \lesssim 
        \|\nu\|_{\ell^{\beta,1}}, 
    \]
where 
    \[\frac{1}{p} = \frac{\theta}{p_0} + \frac{1-\theta}{p_1}, \quad \frac{1}{\beta} = \frac{\theta}{\beta_0} + \frac{1-\theta}{\beta_1}, \quad \theta = \frac{\varepsilon_1}{\varepsilon_0+\varepsilon_1}.\] 
\end{lem}

\begin{prop}\label{prop:global-weak}
    Suppose that $(\frac{1}{q},\frac{1}{p})$ lies in the interior of the triangle $O A_{d/2} E_{d/2}$.
    If $s = \gamma_{\alpha}(p,q)$ and $\beta = \beta_{d/2}(p,q)$, then for all orthonormal systems $(f_j)_j$ in $H^s(M)$ and all sequences $\nu = (\nu_j)_j \in \ell^{\beta,1}$, we have
    \[
        \Big\|\sum_j  \nu_j |e^{it\Delta^{\alpha/2}} f_j|^2\Big\|_{L_t^{\frac{p}{2},\infty}L_x^{\frac{q}{2}}(I\times M)} \lesssim 
        \|\nu\|_{\ell^{\beta,1}}.
    \]
\end{prop}

\begin{proof}
    We choose $\delta>0$ sufficiently small so that $\frac{1}{p_0}=\frac{1}{p}+\delta$, $\frac{1}{p_1}=\frac{1}{p}-\delta$, and both $(\frac{1}{q},\frac{1}{p_0})$ and $(\frac{1}{q},\frac{1}{p_1})$ lie in the interior of the triangle $O A_{d/2} E_{d/2}$.
    For $i=0,1$, define $\varepsilon_i = (-1)^i (2s - 2\gamma_{\alpha}(p_i,q))$. Then, by Proposition \ref{prop:local}, we obtain
    \[
        \Big\|\sum_j  \nu_j |e^{it\Delta^{\alpha/2}} P_l \langle D \rangle^{-s} f_j|^2\Big\|_{L_t^{\frac{p_i}{2},\infty}L_x^{\frac{q}{2}}(I\times M)} \lesssim 
        2^{(-1)^{i+1}\varepsilon_i l}\|\nu\|_{\ell^{\beta_i}}
    \]
    for all orthonormal systems $(f_j)_j$ in $L^2(M)$ and all sequences $\nu = (\nu_j)_j \in \ell^{\beta_i}$ with $\beta_i = \beta_{d/2}(p_i,q)$.
    On the other hand, Theorem \ref{sig0} provides the required uniform boundedness of the sequence of functions $(e^{it\Delta^{\alpha/2}} \langle D \rangle^{-s} f_j)_j$.
    Note that when $s = \gamma_{\alpha}(p,q)$, we have $\varepsilon_0 = \varepsilon_1 > 0$. Thus Lemma \ref{lem2} yields
    \[
        \Big\|\sum_j  \nu_j |e^{it\Delta^{\alpha/2}} \langle D \rangle^{-s} f_j|^2\Big\|_{L_t^{\frac{p}{2},\infty}L_x^{\frac{q}{2}}(I\times M)} \lesssim 
        \|\nu\|_{\ell^{\beta,1}}
    \]
    for all orthonormal systems $(f_j)_j$ in $L^2(M)$ and all sequences $\nu = (\nu_j)_j \in \ell^{\beta,1}$ with $\beta = \beta_{d/2}(p,q)$.
\end{proof}

\begin{rem}
Proposition \ref{prop:global-weak} also holds for the interior of the triangle $OA_{d/2}B$. On the sharp line boundary $(A_{d/2},B)$, a loss of regularity is inevitable for the global estimate, since Lemma \ref{lem2} cannot apply to the oblique boundary.
\end{rem}

\begin{proof}[\textbf{Proof of Theorem \ref{thm:schodinger}.}]
For any fixed $(\frac{1}{q},\frac{1}{p})$ lying in the interior of the triangle $O A_{d/2} E_{d/2}$,
choose two points $(\frac{1}{q},\frac{1}{p_0})$ and $(\frac{1}{q},\frac{1}{p_1})$ also lying in the interior of the triangle $O A_{d/2} E_{d/2}$ such that $\frac{1}{p}=\frac{\theta}{p_0}+\frac{1-\theta}{p_1}$, $\frac{1}{q}=\frac{\theta}{q_0}+\frac{1-\theta}{q_1}$ for some $\theta\in(0,1)$.
From Proposition \ref{prop:global-weak}, for $i=0,1$, we have
\[
    \Big\|\sum_j  \nu_j |e^{it\Delta^{\alpha/2}} f_j|^2\Big\|_{L_t^{\frac{p_i}{2},\infty}L_x^{\frac{q_i}{2}}(I\times M)} \lesssim 
    \|\nu\|_{\ell^{\beta_i,1}}
\]
for all orthonormal systems $(f_j)_j$ in $H^{s_i}(M)$ with $s_i=\gamma_{\alpha}(p_i,q_i)$, and all sequences $\nu=(\nu_j)_j\in \ell^{\beta_i,1}$ with $\beta_i=\beta_{d/2}(p_i,q_i)$.
Moreover, real interpolation yields 
\[(L^{\frac{p_0}{2},\infty}_t L^{\frac{q_0}{2}}_x, L^{\frac{p_1}{2},\infty}_t L^{\frac{q_1}{2}}_x)_{\theta,\frac{p}{2}} = L^{\frac{p}{2}}_t L^{\frac{q}{2},\frac{p}{2}}_x,
\qquad
(\ell^{\beta_0,1}, \ell^{\beta_1,1})_{\theta,\frac{p}{2}} = \ell^{\beta_{d/2}(p,q),\frac{p}{2}}.\]
Note that $(\frac{1}{q},\frac{1}{p})$ lies in the interior of the triangle $O A_{d/2} E_{d/2}$,
hence $p<q$ and $\beta_{d/2}(p,q)<\frac{p}{2}$, 
and thus 
\[L^{\frac{p}{2}}_t L^{\frac{q}{2},\frac{p}{2}}_x \subset L^{\frac{p}{2}}_t L^{\frac{q}{2}}_x, \qquad \ell^{\beta_{d/2}(p,q)}\subset \ell^{\beta_{d/2}(p,q),\frac{p}{2}}.\]
Thus we obtain the estimate
\begin{equation}\label{local6}
    \Big\|\sum_j  \nu_j |e^{it\Delta^{\alpha/2}} f_j|^2\Big\|_{L_t^{\frac{p}{2}}L_x^{\frac{q}{2}}(I\times M)} \lesssim 
    \|\nu\|_{\ell^{\beta}}
\end{equation}
which holds for all orthonormal systems $(f_j)_j$ in $H^s(M)$ with $s=\gamma_{\alpha}(p,q)$, and all sequences $\nu=(\nu_j)_j\in \ell^{\beta}$ with $\beta=\beta_{d/2}(p,q)$.
Interpolating between the estimates \eqref{local6} for points lying in the interior of the triangle $O A_{d/2} E_{d/2}$ 
and the trivial estimate at the endpoint $B$, which follows from Theorem \ref{sig0} and Minkowski's inequality, we obtain that \eqref{local6} holds for all points in the interior of $O A_{d/2} B$ and all orthonormal systems $(f_j)_j$ in $H^s(M)$ with $s \ge \gamma_{\alpha}(p,q)$.

Finally, interpolating between the points on the line segment $(C_{d/2}, D)$ 
and the points inside the triangle $O A_{d/2} B$ that approach the line segment $(O, A_{d/2})$ arbitrarily closely, 
we complete the proof of Theorem \ref{thm:schodinger}.
\end{proof}

\subsection{Wave and Klein-Gordon case}
We have a similar proof for the wave and Klein-Gordon propagators.  

\begin{prop}\label{localw1}
    Let $d \geq 2$, $m \geq 0$. 
 Then for all $\frac{d-1}{2}$-admissible pairs $(p,q)$ such that $(\frac{1}{q},\frac{1}{p})$ lies in the interior of the triangle $O A_{\frac{d-1}{2}} E_{\frac{d-1}{2}}$ and for all $l\in\mathbb Z$, we have
    \begin{equation}\label{localw2}
        \Big\|\sum_j  \nu_j |e^{it\sqrt{\Delta + m^2}}P_l f_j|^2\Big\|_{L_t^{p/2}L_x^{q/2}(I\times M)} \lesssim
        2^{2l\gamma(p,q)}\|\nu\|_{\ell^\beta}
    \end{equation}
 for all orthonormal systems $(f_j)_j$ in $L^2(M)$ and all sequences $\nu=(\nu_j)_j\in \ell^\beta$ with $\beta=\beta_{\frac{d-1}{2}}(p,q)$.
\end{prop}

\begin{proof}
    The proof follows the same lines as that of Proposition \ref{prop:local}.
    The only change is that Lemma~\ref{din} is replaced by Lemma~\ref{dislem}, the  dispersive estimate for the wave/Klein--Gordon equations.

    The local Weyl estimate implies
    the $\infty$-endpoint estimate
    \begin{equation}\label{localw3}
        \Big\|\sum_j  \nu_j |e^{it\sqrt{\Delta + m^2}}P_l f_j|^2\Big\|_{L_t^{\infty}L_x^{\infty}(I\times M)} \lesssim
        2^{ld}\|\nu\|_{\ell^\infty}.
    \end{equation}

    Next, we consider the segment $(A_{\frac{d-1}{2}}, E_{\frac{d-1}{2}}]$.
    For any small $\eps>0$, we introduce a holomorphic family of operators $T^\varepsilon_{z,l}$ in the strip $\{z\in\mathbb{C}: -r/2\le \operatorname{Re}z\le 0\}$ whose kernel is given by
    \[K^\eps_{z,l}(t,x,s,y)=\1_{\eps< |t-s|}(t-s)^{-1-z}\sum_{k} \varphi(2^{-l}\lambda_k)^2 e^{i (t-s)\sqrt{\lambda_k^2+m^2}}e_k(x)\overline{e_k(y)}.  \]
 We apply Lemma \ref{dua} and Stein's complex interpolation.  Set $r=\beta'$ and let $z_1=-\frac r2+ib$, $z_2=ib$ with $b\in\mathbb{R}$. 
Using Lemma \ref{dislem} together with the Hardy–Littlewood–Sobolev inequality, we have
\[\|WT^\eps_{z_1,l}  \overline{W}\|_{\mathfrak S^{2}}\lesssim 2^{\frac{(d+1)l}{2}} \|W\|_{L_t^{\frac4{r-d+1}}L_x^2}^2.\]
Interpolating with the same $\mathfrak S^\infty$-bound as \eqref{eq:infinitybound} yields
\[\|WT^\eps_{-1,l} \overline{W}\|_{\mathfrak S^{r}}\lesssim   2^{\frac2p\frac{d+1}{d-1}l} \|W\|_{L_t^{\frac{2r}{r-d+1}}L_x^{r}(M)}^2
\]
for all $W\in L_t^{\frac{2r}{r-d+1}}L_x^{r}(I\times M)$.
    Consequently, letting $\varepsilon \to 0$, Lemma \ref{dua} implies
    \begin{equation}\label{localw4}
        \Big\|\sum_j  \nu_j |e^{it\sqrt{\Delta + m^2}}P_l f_j|^2\Big\|_{L_t^{p/2}L_x^{q/2}(I_{l,n}\times M)} \lesssim 2^{\frac2p\frac{d+1}{d-1}l} \|\nu\|_{\ell^\beta}.
    \end{equation}
  Then \eqref{localw2} follows by complex interpolation between \eqref{localw3} and \eqref{localw4}. 
\end{proof}

\begin{prop}\label{prop01}
    Suppose that $(\frac{1}{q},\frac{1}{p})$ lies in the interior of the triangle $O A_{\frac{d-1}{2}} E_{\frac{d-1}{2}}$.
    If $s= \gamma(p,q)$ and $\beta=\beta_{\frac{d-1}{2}}(p,q)$, then for all orthonormal systems $(f_j)_j$ in $H^s(M)$ and all sequences $\nu=(\nu_j)_j\in \ell^{\beta,1}$ we have
    \begin{equation}\label{eq:BH1}
        \Big\|\sum_j  \nu_j |e^{it\sqrt{\Delta + m^2}} f_j|^2\Big\|_{L_t^{\frac{p}{2},\infty}L_x^{\frac{q}{2}}(I\times M)} \lesssim 
        \|\nu\|_{\ell^{\beta,1}}.
    \end{equation}
\end{prop}

\begin{proof}
    We choose $\delta>0$ sufficiently small so that $\frac{1}{p_0}=\frac{1}{p}+\delta$, $\frac{1}{p_1}=\frac{1}{p}-\delta$, and both $(\frac{1}{q},\frac{1}{p_0})$ and $(\frac{1}{q},\frac{1}{p_1})$ lie in the interior of the triangle $OA_{\frac{d-1}{2}}E_{\frac{d-1}{2}}$.
    For $i=0,1$, define $\varepsilon_i = (-1)^i (2s-2\gamma(p_i,q))$. Hence, by Proposition \ref{localw1}, for all orthonormal systems $(f_j)_j$ in $L^2(M)$ and all sequences $\nu=(\nu_j)_j\in \ell^{\beta_i}$ with $\beta_i=\beta_{\frac{d-1}{2}}(p_i,q)$, we have
    \[
        \Big\|\sum_j \nu_j |e^{it\sqrt{\Delta + m^2}} P_l \langle D \rangle^{-s} f_j|^2\Big\|_{L_t^{\frac{p_i}{2},\infty} L_x^{\frac{q}{2}}(I\times M)} \lesssim 
        2^{(-1)^{i+1}\varepsilon_i l}\|\nu\|_{\ell^{\beta_i}}.
    \]
    On the other hand, Theorem \ref{sig1} provides the required uniform boundedness for the function sequence $(e^{it\sqrt{\Delta + m^2}} \langle D \rangle^{-s} f_j)_j$.
    Note that when $s=\gamma(p,q)$, we have $\varepsilon_0 = \varepsilon_1 > 0$. Thus Lemma \ref{lem2} yields, for all orthonormal systems $(f_j)_j$ in $L^2(M)$ and all sequences $\nu=(\nu_j)_j\in \ell^{\beta,1}$ with $\beta=\beta_{\frac{d-1}{2}}(p,q)$, the inequality
    \[
        \Big\|\sum_j \nu_j |e^{it\sqrt{\Delta + m^2}} \langle D \rangle^{-s} f_j|^2\Big\|_{L_t^{\frac{p}{2},\infty} L_x^{\frac{q}{2}}(I\times M)} \lesssim 
        \|\nu\|_{\ell^{\beta,1}},
    \]
 which is just the desired estimate \eqref{eq:BH1}.
\end{proof}

\begin{proof}[\textbf{Proof of Theorem \ref{ns2}}]
For any fixed $(\frac{1}{q},\frac{1}{p})$ in the interior of $O A_{\frac{d-1}{2}} E_{\frac{d-1}{2}}$,
choose two points $(\frac{1}{q},\frac{1}{p_0})$ and $(\frac{1}{q},\frac{1}{p_1})$ also in the interior of $O A_{\frac{d-1}{2}} E_{\frac{d-1}{2}}$ such that $\frac{1}{p}=\frac{\theta}{p_0}+\frac{1-\theta}{p_1}$, $\frac{1}{q}=\frac{\theta}{q_0}+\frac{1-\theta}{q_1}$ for some $\theta\in(0,1)$.
From Proposition \ref{prop01}, for $i=0,1$, we have
\[
    \Big\|\sum_j  \nu_j |e^{it\sqrt{\Delta + m^2}} f_j|^2\Big\|_{L_t^{\frac{p_i}{2},\infty}L_x^{\frac{q_i}{2}}(I\times M)} \lesssim 
    \|\nu\|_{\ell^{\beta_i,1}}
\]
for all orthonormal systems $(f_j)_j$ in $H^{s_i}(M)$ with $s_i=\gamma(p_i,q_i)$ and all sequences $\nu=(\nu_j)_j\in \ell^{\beta_i,1}$ with $\beta_i=\beta_{\frac{d-1}{2}}(p_i,q_i)$.
Moreover, real interpolation yields \[
(L^{\frac{p_0}{2},\infty}_t L^{\frac{q_0}{2}}_x, L^{\frac{p_1}{2},\infty}_t L^{\frac{q_1}{2}}_x)_{\theta,\frac{p}{2}} = L^{\frac{p}{2}}_t L^{\frac{q}{2},\frac{p}{2}}_x,
\qquad
(\ell^{\beta_0,1}, \ell^{\beta_1,1})_{\theta,\frac{p}{2}} = \ell^{\beta_{\frac{d-1}{2}}(p,q),\frac{p}{2}}.\]
Note that $(\frac{1}{q},\frac{1}{p})$ lies in the interior of the triangle $O A_{\frac{d-1}{2}} E_{\frac{d-1}{2}}$,
hence $p<q$, $\beta_{\frac{d-1}{2}}(p,q)<\frac{p}{2}$, 
and thus $$L^{\frac{p}{2}}_t L^{\frac{q}{2},\frac{p}{2}}_x \subset L^{\frac{p}{2}}_t L^{\frac{q}{2}}_x, \qquad \ell^{\beta_{\frac{d-1}{2}}(p,q)}\subset \ell^{\beta_{\frac{d-1}{2}}(p,q),\frac{p}{2}}.$$
We obtain
\begin{equation}\label{localw6}
    \Big\|\sum_j  \nu_j |e^{it\sqrt{\Delta + m^2}} f_j|^2\Big\|_{L_t^{\frac{p}{2}}L_x^{\frac{q}{2}}(I\times M)} \lesssim 
    \|\nu\|_{\ell^{\beta}}
\end{equation}
for all orthonormal systems $(f_j)_j$ in $H^s(M)$ with $s=\gamma(p,q)$, and all sequences $\nu=(\nu_j)_j\in \ell^{\beta}$ with $\beta=\beta_{\frac{d-1}{2}}(p,q)$.
Interpolating between the estimates \eqref{localw6} for points  in the interior of $O A_{\frac{d-1}{2}} E_{\frac{d-1}{2}}$ 
and the sharp line estimate at the trivial endpoint $B$, which follows from Theorem \ref{sig1} and Minkowski's inequality, we obtain \eqref{localw6} for all points in the interior of $O A_{\frac{d-1}{2}} B$ and all orthonormal systems $(f_j)_j$ in $H^s(M)$ with $s \ge \gamma(p,q)$.

Finally, interpolating between points on the line segment $(C_{\frac{d-1}{2}}, D)$ and points inside the triangle $O A_{\frac{d-1}{2}} B$ that approach the line segment $(O, A_{\frac{d-1}{2}})$ arbitrarily closely,
we complete the proof of Theorem \ref{ns2}.
\end{proof}

\subsection{Comparison of the two proofs}
The main difference between Sections \ref{sec:proof1} and \ref{sec:proof2} lies in how one obtains the Lorentz-norm estimates in the interior of the triangle $O A_\sigma E_\sigma$, while the interpolation arguments are essentially the same.
In Section \ref{sec:proof1}, these estimates follow from the sharp line estimates of Theorem \ref{thm:sharp} and the Lieb-Sobolev inequality from Lemma \ref{lemma}. There is an $\epsilon$-loss of regularity resulting from the transition from \emph{sharp to non-sharp} estimates. 
In Section \ref{sec:proof2}, we instead prove frequency-localized estimates directly via the Schatten duality from Lemma \ref{dua} and then use the globalization argument in Lemma \ref{lem2}. One advantage is that there is no loss of regularity.

On the other hand, with the method in Section \ref{sec:proof1}, we cannot attain $\beta=p/2$ in the interior of $O A_\sigma C_\sigma D$, because it also happens on the sharp line.
The key problem is that the global Strichartz estimate fails on the critical line $[O, A_\sigma]$. We demonstrate for the fractional Schr\"odinger case.

\begin{example}\label{ex}
    Assume \((\frac{1}{q},\frac{1}{p})\) lies on the critical line $[O, A_{d/2}]$ and $\alpha>1$. 
    Let 
    \[s=\gamma_\alpha(p,q)= \frac d2-\frac dp, \qquad \beta=\beta_{d/2}(p,q)=\frac p2.\]
  Let $(e_k)_k$ be an orthonormal eigenbasis in $L^2(M)$ associated with the (increasing) eigenvalues $(\la_k)_k$ of $\sqrt{\Delta}$.
Define the dyadic subspace 
    \[
        \mathcal E_l
        =
        \operatorname{span}\{e_k:N_l\leq \lambda_k<2N_l\}, \qquad N_l=2^l,\,l\in \mathbb Z,
    \]
with dimension $D_l=\dim \mathcal E_l$.
    By the Weyl law and the local Weyl law, we have for all \(x\in M\) and all sufficiently large \(l\)
    \begin{equation}\label{ex}
        D_l\approx N_l^d, \qquad
        \sum_{N_l\leq \lambda_k<2N_l}|e_k(x)|^2
        \gtrsim N_l^d,
    \end{equation}
where the second inequality indicates that the kernel of the spectral projector onto
    \(\mathcal E_l\) has a uniform lower bound in $d$-power of the spectrum.
    Given any large integer $L$, for each $l$ such that \(1\leq l\leq L\), choose an orthonormal basis $\{e_{l,m}\}_{m=1}^{D_l}$
    of \(\mathcal E_l\) and define
    \[
        e_{l,m,s}=\langle D\rangle^{-s}e_{l,m},
        \qquad
        1\leq l\leq L,\quad 1\leq m\leq D_l.
    \]
    Then \((e_{l,m,s})_{l\ge 1,1\le m\le D_l}\) forms an orthonormal basis of \(H^s(M)\).
We assign the same coefficient $\nu_{l,m}=N_l^{-2d/p}$ in the \(l\)-th dyadic block.
    Then by the first estimate in \eqref{ex}
    \[
        \|\nu\|_{\ell^{p/2}}^{p/2}
        =
        \sum_{l=1}^L\sum_{m=1}^{D_l} |\nu_{l,m}|^{p/2}
        =
        \sum_{l=1}^L D_l N_l^{-d}
        \approx
               L.
    \]
    Note that 
    \[|e^{it\Delta^{\alpha/2}}e_{l,m,s}(x)|^2 = |e^{it \lambda_{l,m}^\alpha} (1+\lambda_{l,m}^2)^{-s/2} e_{l,m}|^2 \approx N_l^{-2s}|e_{l,m}(x)|^2.\]    
Using the second estimate in \eqref{ex}, we have
\[s_L(t,x) \gtrsim \sum_{l=1}^L N_l^{-2d/p}N_l^{-2s} \sum_{m=1}^{D_l}|e_{l,m}(x)|^2 \gtrsim \sum_{l=1}^L N_l^{-2d/p}N_l^{-2s}N_l^d=L,\]
  where   
  \[s_L(t,x)=\sum_{l=1}^L\sum_{m=1}^{D_l}\nu_{l,m}|e^{it\Delta^{\alpha/2}}e_{l,m,s}(x)|^2.\]
If the Strichartz estimate \eqref{eq:str} holds, then we have
    $$\|s_L\|_{L_t^{p/2}L_x^{q/2}(I\times M)} \lesssim \|\nu\|_{\ell^{p/2}},$$
    which implies $L \lesssim L^{2/p}$. 
We arrive at a contradiction when $L$ is sufficiently large, 
  since $p>2$ as $(\frac{1}{q},\frac{1}{p})$ lies on the critical line $[O, A_{d/2}]$.  \end{example}

In contrast, using the method in Section \ref{sec:proof2}, if we have the frequency-localized Strichartz estimates on the critical line $[O, A_\sigma)$, then we can further improve the range of $\beta$ to $\beta = \frac{p}{2}$ in the interior of $O A_\sigma C_\sigma D$.

\begin{prop}\label{prop:beta=p/2}
    Let $d \geq 3$, $\alpha \in (0,\infty)\setminus\{1\}$. 
    Assume that for all $\frac{d}{2}$-admissible pairs $(p,q)$ such that $(\frac{1}{q},\frac{1}{p})$ lies on the critical line $[O, A_{d/2})$, 
    we have 
    \begin{equation}\label{eq:local7}
        \Big\|\sum_j  \nu_j |e^{it\Delta^{\alpha/2}}P_l f_j|^2\Big\|_{L_t^{p/2}L_x^{q/2}(I\times M)} \lesssim
        2^{2l\gamma_{\alpha}(p,q)}\|\nu\|_{\ell^{p/2}}
    \end{equation}
for all orthonormal systems $(f_j)_j$ in $L^2(M)$ and all sequences $\nu=(\nu_j)_j\in \ell^{p/2}$.
    Then for all $\frac{d}{2}$-admissible pair $(p,q)$ such that $(\frac{1}{q},\frac{1}{p})$ lies in the interior of $O A_{d/2} C_{d/2} D$, we have
    \begin{equation}\label{eq:global-p/2-OACD}
        \Big\|\sum_j  \nu_j |e^{it\Delta^{\alpha/2}} f_j|^2\Big\|_{L_t^{{p}/{2}}L_x^{{q}/{2}}(I\times M)} \lesssim 
        \|\nu\|_{\ell^{p/2}}
    \end{equation}
for all orthonormal systems $(f_j)_j$ in $H^s(M)$ with $s\ge \gamma_{\alpha}(p,q)$ and all sequences $\nu=(\nu_j)_j\in \ell^{p/2}$.
\end{prop}

\begin{proof}
    If $(\frac{1}{q},\frac{1}{p})$ lies in the line segment $[D,C_{d/2}]$, 
    then by Theorem \ref{sig0} and the triangle inequality,
   \eqref{eq:local7} holds with $\frac{p}{2}=1$.
By the assumed  estimate on the critical line $[O, A_{d/2})$ and applying complex interpolation,
    we obtain that \eqref{eq:local7} holds for all points in the interior of the triangle $O A_{d/2} C_{d/2} D$.
By Lemma~\ref{lem2} and repeating the proof of Proposition~\ref{prop:global-weak},  when $(\frac1q,\frac1p)$ lies in the interior of $O A_{d/2} C_{d/2} D$, we have
    \[
        \Big\|\sum_j  \nu_j |e^{it\Delta^{\alpha/2}} f_j|^2\Big\|_{L_t^{\frac{p}{2},\infty}L_x^{\frac{q}{2}}(I\times M)} \lesssim 
        \|\nu\|_{\ell^{\beta,1}}
    \]
for all orthonormal systems $(f_j)_j$ in $H^s(M)$ with $s\ge\gamma_{\alpha}(p,q)$ and all sequences $\nu=(\nu_j)_j\in \ell^{\beta,1}$ with $\beta=\frac{p}{2}$.
Finally, using real interpolation as in the proof of Theorem \ref{thm:schodinger}, we obtain the desired estimate \eqref{eq:global-p/2-OACD} for all points in the interior of $O A_{d/2} C_{d/2} D$.
\end{proof}

\begin{rem}
We are currently unable to obtain the frequency-localized Strichartz estimates \eqref{eq:local7} on the critical line $[O, A_\sigma)$ for general compact manifolds. 
Therefore, it remains open whether the exponent $\beta=p/2$ can be attained in the interior of $O A_\sigma C_\sigma D$. However, we can jump this to prove a much better exponent on the flat torus exploring its particular geometry. See Corollary \ref{cor:OACD}.
\end{rem}

\section{Improvements on the flat torus}\label{sec:impr}
Using decoupling inequalities, one can obtain improved diagonal (single-function) Strichartz estimates for fractional Schr\"odinger equations on the flat torus, see \cite[Theorem 7]{wang2025strichartz}.
Combining these single-function estimates with interpolation further refines the Strichartz estimates for sharp admissible pairs of exponents on the flat torus.
Proposition \ref{prop:sub-torus} and Proposition \ref{prop:super-torus} below correspond respectively to \cite[Corollary 1 and Corollary 2]{wang2025strichartz}.
We first consider the subcritical regime.
\begin{prop}\label{prop:sub-torus}
     Let $d\ge 1,\,\alpha \ge2$. Suppose $2\le q\le\frac{2(d+2)}{d}$ and $\frac1p=\frac d2(\frac12-\frac1q)$. 
	 Then for all $\gamma \in (0,d/2]$ and all $s> \gamma$,  we have
\begin{equation}\label{eq:prop-subsharp}
    \Big\|\sum_j  \nu_j |e^{it\Delta^{\alpha/2}} f_j|^2\Big\|_{L^{p/2}_{t}L_x^{q/2}(\mathbb T^{d+1})} \lesssim 
    \|\nu\|_{\ell^\beta},  \qquad \beta< \frac{d}{d-2\gamma}\,,
    \end{equation}
for  all orthonormal systems $(f_j)_j\subset H^s(\mathbb T^d)$ and
    all sequences $\nu=(\nu_j)_j\in \ell^\beta$.
\end{prop}
Interpolating with the above estimate on the sharp segment $[E_{d/2}, B_{d/2}]$, we can further improve Theorem \ref{thm:schodinger} in the interior of the triangle $O E_{d/2} B_{d/2} $.

\begin{cor}\label{cor0}
    Let $d\ge 1$, $\alpha \ge 2$. Suppose $(\frac{1}{q},\frac{1}{p})$ belongs to the interior of the triangle $O E_{d/2} B_{d/2}$.
    Then for all $\gamma \in \left(\frac d2-\frac dq-\frac2p,\,\frac d2\right]$ and all $s>\gamma$, 
    the Strichartz estimate \eqref{eq:prop-subsharp} holds
    for all orthonormal systems $(f_j)_j\subset H^s(\mathbb T^d)$ and all sequences $\nu=(\nu_j)_j\in \ell^\beta$. 
\end{cor}

\begin{proof}
    For a point $(\frac{1}{q},\frac{1}{p})$ in the interior of the triangle $O E_{d/2} B_{d/2}$,
    we can find a point $X_0=(\frac{1}{q_0},\frac{1}{p_0})$ on the segment $(E_{d/2}, B_{d/2})$ such that $\frac{1}{q}=\frac{\theta}{q_0}$ and $\frac{1}{p}=\frac{\theta}{p_0}$, where $\theta= \frac{2(dp+2q)}{dpq} \in (0,1)$.
    Then for the point $X_0$, the Strichartz estimate \eqref{eq:prop-subsharp} 
    holds for all $s > \gamma_0$, where $\gamma_0 \in (0,d/2]$, by Proposition \ref{prop:sub-torus}.
    On the other hand, consider a point $X_1=(\frac{1}{q_1},\frac{1}{p_1})$ in the interior of the triangle $O E_{d/2} B_{d/2}$ lying on the segment $(O, X_0)$ and arbitrarily close to the origin $O$.
    Then, as $X_1 \to O$, we have $\gamma_{\alpha}(p_1,q_1) \to d/2$ and $\beta_{d/2}(p_1,q_1) \to \infty$, and by Theorem \ref{thm:schodinger}, the corresponding orthonormal Strichartz estimate holds for all $s \ge \gamma_{\alpha}(p_1,q_1)$ and $\beta \le \beta_{d/2}(p_1,q_1)$.
    Interpolating the above two estimates, we obtain the desired estimate for all
    \[s>\gamma=\theta \gamma_0+(1-\theta)\frac{d}{2} \in \left(\frac d2-\frac dq-\frac2p,\,\frac{d}{2}\right], \qquad \frac{1}{\beta}>\frac{\theta}{d/(d-2\gamma_0)} 
    =\frac{d-2\gamma}{d}.\]
\end{proof}

\begin{rem}
When we choose $\gamma>\gamma_\alpha(p,q)$, with loss of some regularity,  
Corollary \ref{cor0} indeed improves Theorem \ref{thm:schodinger} in the subcritical regime as the system exponent 
\[\frac{d}{d-2\gamma}>\beta_{d/2}(p,q).\] 
\end{rem}
Next we consider the improvement in the supercritical regime.
\begin{prop}\label{prop:super-torus}
    Let $d\ge5$, $\alpha\ge2$. Suppose $\frac{2(d+1)}{d-1}< q\le \frac{2d}{d-2}$ with $\frac1p=\frac d2(\frac12-\frac1q)$. Then for all $s>\gamma_{\alpha}(p,q)$, we have
    \begin{equation}\label{propeq}
    \Big\|\sum_j  \nu_j |e^{it\Delta^{\alpha/2}} f_j|^2\Big\|_{L_t^{p/2}L_x^{q/2}(\mathbb T^{d+1})} \lesssim 
    \|\nu\|_{\ell^\beta}, \qquad \beta<\frac{pd(d-3)}{8+pd(d-4)},
    \end{equation}
for all orthonormal systems $(f_j)_j\subset H^s(\mathbb T^d)$ and all sequences $\nu=(\nu_j)_j\in \ell^\beta$.
\end{prop}

Theorem \ref{thm:schodinger} can be improved in the supercritical regime, namely the interior of \(OA_{d/2}C_{d/2}D\), by interpolation with the above estimate on the sharp supercritical segment $(A_{d/2},C_{d/2}]$. 
For every point, the interpolation is quite flexible, and a straightforward calculation shows that the optimal choice is obtained by taking the line that passes through the point \((1/q,1/p)\) and the segment $(A_{d/2},C_{d/2}]$ with the \emph{best ratio} (i.e., relatively closest to the segment $(A_{d/2},C_{d/2}]$). The optimal line is determined by the origin $O$ for the region $OA_{d/2}C_{d/2}$, while for region $OC_{d/2}D$ it is determined by the Keel-Tao endpoint $C_{d/2}$. See Figure \ref{fig1}.
Therefore, combining Theorem \ref{thm:schodinger} and Proposition \ref{prop:super-torus}, we obtain the following corollary.

\begin{cor}\label{cor:OACD}
    Let $d\ge5$, $\alpha\ge2$.  Suppose $(p,q)$ is non-sharp $\frac{d}{2}$-admissible. \\    
 \rm{(I)} If $(\frac{1}{q},\frac{1}{p})$ lies in the interior of $O A_{d/2} C_{d/2}$. Then for all $s>\gamma_{\alpha}(p,q)$, we have
    \[
    \Big\|\sum_j  \nu_j |e^{it\Delta^{\alpha/2}} f_j|^2\Big\|_{L_t^{p/2}L_x^{q/2}(\mathbb T^{d+1})} \lesssim 
    \|\nu\|_{\ell^\beta}, \qquad \beta<\frac{d(d-3)pq}{2d(d-4)p+4(d-2)q},
    \]
for all orthonormal systems $(f_j)_j\subset H^s(\mathbb T^d)$ and all sequences $\nu=(\nu_j)_j\in \ell^\beta$.\\    
\rm{(II)}    If $(\frac{1}{q},\frac{1}{p})$ lies in the interior of $O C_{d/2} D$. Then for all $s\ge\gamma_{\alpha}(p,q)$, we have
    \[
    \Big\|\sum_j  \nu_j |e^{it\Delta^{\alpha/2}} f_j|^2\Big\|_{L_t^{p/2}L_x^{q/2}(\mathbb T^{d+1})} \lesssim 
    \|\nu\|_{\ell^\beta}, \qquad  \beta < \frac{(d-2)(d-3)pq}{2(d-2)(d-3)q-2(d-4)p},
    \]
    for all orthonormal systems $(f_j)_j\subset H^s(\mathbb T^d)$ and all sequences $\nu=(\nu_j)_j\in \ell^\beta$.

\end{cor}

\begin{proof}
    (I) For point $(\frac{1}{q},\frac{1}{p})$ in the interior of the triangle $O A_{d/2} C_{d/2}$, we can find a point $X_0=(\frac{1}{q_0},\frac{1}{p_0})$ on the segment $(A_{d/2}, C_{d/2})$ such that $\frac{1}{q}=\frac{\theta}{q_0}$, $\frac{1}{p}=\frac{\theta}{p_0}$, where $\theta= \frac{2(dp+2q)}{dpq} \in (0,1)$.
    Then for point $X_0$, the Strichartz estimate \eqref{propeq} holds for all $s>\gamma_{\alpha}(p_0,q_0)$ and the corresponding $\beta$ by Proposition \ref{prop:super-torus}.
    On the other hand, we consider the point $X_1=(\frac{1}{q_1},\frac{1}{p_1})$ in the interior of the triangle $O A_{d/2} C_{d/2}$ arbitrarily close to the origin $O$, and lying on the segment $(O, X_0)$.
    Then for point $X_1 \rightarrow O$, the corresponding orthonormal Strichartz estimate holds for all $s\ge\gamma_{\alpha}(p_1,q_1) \rightarrow d/2$ and $\beta\le p_1/2 \rightarrow \infty$ by Theorem \ref{thm:schodinger}.
    Interpolating the above two estimates, we obtain the desired estimate for all $s>\gamma_{\alpha}(p,q)$ and 
    $$\frac{1}{\beta}>\frac{\theta[8+p_0d(d-4)]}{p_0d(d-3)} 
    =\frac{2d(d-4)p+4(d-2)q}{d(d-3)pq} .$$

    (II) For a point $(\frac{1}{q},\frac{1}{p})$ in the interior of the triangle $O C_{d/2} D$, 
    we can find a point $X_2=(\frac{1}{q_2},\frac{1}{p_2})=(0,\frac{(d-2)q-dp}{(d-2)pq-2dp})$ on the segment $(O, D)$ such that $\frac{1}{q}=\frac{1-\theta}{q_2}+\frac{\theta}{q_c}$, $\frac{1}{p}=\frac{1-\theta}{p_2}+\frac{\theta}{p_c}$, 
    where $(\frac{1}{q_c},\frac{1}{p_c})=C_{d/2}$ and $\theta=\frac{2d}{(d-2)q} \in (0,1)$.
    For Keel-Tao endpoint $C_{d/2}$, the Strichartz estimate \eqref{propeq} holds for all $s\ge \gamma_{\alpha}(p_c,q_c)$ and the corresponding $\beta$ by Proposition \ref{prop:super-torus}.
    On the other hand, consider point $X_3=(\frac{1}{q_3},\frac{1}{p_3})$ lies on the segment $(X_2, C_{d/2})$ and is arbitrarily close to $X_2$.
    For point $X_3\rightarrow X_2$, the corresponding orthonormal Strichartz estimate holds for all $s\ge\gamma_{\alpha}(p_3,q_3) \rightarrow \gamma_{\alpha}(p_2,q_2)$ and $\beta<p_3/2\rightarrow p_2/2$ by Theorem \ref{thm:schodinger}.
    Interpolating the above two estimates, we obtain the desired estimate for all $s\ge\gamma_{\alpha}(p,q)$ and
    $$\frac{1}{\beta}>\frac{1-\theta}{p_2/2}+\frac{\theta[8+p_cd(d-4)]}{p_cd(d-3)}=\frac{2(d-2)(d-3)q-2(d-4)p}{(d-2)(d-3)pq}.$$
\end{proof}

\begin{rem}
    Furthermore, \cite{wang2025strichartz} conjectured from the discrete restriction conjecture that the range of $\beta$ in Proposition \ref{prop:super-torus}
    can be further extended to \[\beta<\frac{d+1}{d}=\beta_{d/2}\left(\frac{2(d+1)}d,\frac{2(d+1)}{d-1}\right),\]
    which corresponds to the proved sharp exponent for the critical point $A_{d/2}$.
    If the conjecture holds, then the two estimates in Corollary \ref{cor:OACD} can be improved respectively to 
    \[\beta<\frac{(d+1)pq}{2dp+4q}  \quad  \text{and}\quad \beta<\frac{(d+1)(d-2)pq}{2(d+1)(d-2)q-2dp}.\]
\end{rem}

\appendix \section{Necessary conditions}  \label{app:nec}

In this appendix, we discuss necessary conditions for the range of the exponent \(\beta\). 
The arguments below are formulated in terms of frequency-localized Strichartz estimates.
This is sufficient for our purpose: if a global orthonormal Strichartz estimate holds with Sobolev regularity $s$,
then, after restricting the initial data $(f_j)_j$ to the localized subspace \(\{\lambda_k\le N\}\), we have
\begin{equation}\label{eq:applocal}
    \Big\|\sum_j \nu_j |e^{itP}f_j|^2\Big\|_{L^{p/2}_tL^{q/2}_x(I\times M)}
    \lesssim N^{2s}\|\nu\|_{\ell^\beta}
\end{equation}
Therefore, any necessary condition for the frequency-localized estimate is also a necessary condition for the corresponding global Strichartz estimate.

\subsection{Fractional Schr\"odinger case for $\alpha>1$}
We first consider the sphere $M=\mathbb{S}^d$ with the standard metric.
Let $N\gg1$ and fix a point $x_0\in M$.
It is well known, see e.g. Sogge \cite{sogge2014hangzhou}, that for each eigenvalue $\lambda_j \approx N$ of $\sqrt{\Delta}$,
we can choose an $L^2$-normalized zonal function $Z_j$ concentrating near $x_0$ in the sense that
$$| Z_j(x) | \approx N^{\frac{d-1}{2}} \text{ for } d_g(x,x_0)\lesssim N^{-1}.$$
Moreover, by selecting distinct eigenvalues, the family $(Z_j)_j$ forms an orthonormal system of cardinality $\approx N$.
Then we obtain 
\begin{equation}\label{nec1}
    \Big\|\sum_j   |e^{it\Delta^{\alpha/2}}Z_j|^2\Big\|_{L^{p/2}_tL^{q/2}_x(I\times M)} \gs N^{d} (N^{-d})^{\frac2q}|I|^{\frac2p}
\end{equation}
since $|e^{it\Delta^{\alpha/2}}Z_j|=|Z_j|$. 

Suppose that the frequency-localized Strichartz estimate \eqref{eq:applocal} holds with Sobolev regularity $s$. 
Then we have the upper bound
\[
    N^{2s}\|(1)_j\|_{\ell^\beta}
    \approx N^{2s}N^{1/\beta}.
\]
Comparing this upper bound with \eqref{nec1}, we obtain
\[
    N^d(N^{-d})^{\frac2q}|I|^{\frac2p}
    \lesssim N^{2s}N^{1/\beta}.
\]
Since $|I|\approx1$ in the present case, this implies
\[
    \frac1\beta\ge d-\frac{2d}{q}-2s.
\]
In particular, hence the estimate \eqref{eq:applocal} holds for all $s\geq\gamma_\alpha(p,q)$, we have the necessary condition
\[
    \frac{1}{\beta}\geq d-\frac{2d}{q}-2\gamma_{\alpha}(p,q) = \frac{2}{p},
\]
where in the last equality we used $\alpha>1$ and hence
$\gamma_\alpha(p,q)=\frac d2-\frac dq-\frac1p$.
This corresponds exactly to the range in the interior of $O A_{d/2} C_{d/2} D$ in Theorem \ref{thm:schodinger}.
However, the sharpness of $\beta$ here applies only to the sphere, since we have improved the range of $\beta$ on the flat torus in Section \ref{sec:impr}.

Next, we recall the counting example, which works on any compact manifold.
Suppose that the frequency-localized Strichartz estimate \eqref{eq:applocal} holds with regularity $s$.
We take an orthonormal basis of the spectral subspace $\left\{\lambda_k \le N\right\}$ and $\nu_j=1$, the local Weyl law gives a lower bound of order $N^d$ for the left-hand side, while the right-hand side is bounded by $N^{2s}N^{d/\beta}$.
Therefore
\[
    N^d\lesssim N^{2s}N^{d/\beta},
\]
and hence
\[
    \frac{1}{\beta} \geq 1-\frac{2s}{d}.
\]
When $\alpha>1$, we have the necessary condition
\[
    \frac1\beta\ge 1-\frac{2\gamma_{\alpha}(p,q)}{d}=\frac1{\beta_{d/2}(p,q)},
\]
This gives the sharpness of the range of $\beta$ in the interior of $O A_{d/2}B$ on any compact manifold.

\subsection{Fractional Schr\"odinger case for $\alpha \in (0,1)$}
When $0<\alpha<1$, we follow the long-time frequency-localized examples in \cite{wang2025strichartz}.
In this case, the corresponding frequency-localized estimates are valid on intervals of length $|I|\approx N^{1-\alpha}$, and the necessary examples should be tested on such intervals.

Fix $x_0\in M$, and let $(e_k)_k$ be an orthonormal eigenbasis in $L^2(M)$ associated with the eigenvalues $(\lambda_k)_k$ of $\sqrt{\Delta}$.
By the pointwise Weyl law, one may choose a fixed $\delta>0$ such that, for large $j$,
\[
    c_j:=\sum_{\lambda_k\in(j-\delta,j]} |e_k(x_0)|^2\approx j^{d-1}.
\]
Define
\begin{equation}\label{eq:fj}
    f_j(x)=c_j^{-1/2}\sum_{\lambda_k\in(j-\delta,j]}e_k(x_0)\overline{e_k(x)}.
\end{equation}
Then $\|f_j\|_{L^2(M)}=1$,
and
\(
f_j(x_0)=c_j^{1/2}\approx j^{(d-1)/2}.
\)
Moreover, by the spectral localization of $f_j$, Bernstein's inequality and the mean value theorem imply that,
\[
|f_j(x)|\gtrsim j^{(d-1)/2},~~
\text{for}~~ d_g(x,x_0)\lesssim j^{-1}.
\]
In particular, for $j\approx N$,
\[
    |f_j(x)|\approx N^{\frac{d-1}{2}}~~ \text{for} ~~d_g(x,x_0)\lesssim N^{-1}.
\]
Note that
$$e^{it\Delta^{\alpha/2}}f_j(x)=c_j^{-1/2}\sum_{\lambda_k\in(j-\delta,j]}e_k(x_0)e^{it(\lambda_k^\alpha-j^\alpha)}\overline{e_k(x)}.$$
Since $0<\alpha<1$, for $\lambda_k\in(j-\delta,j]$ and $j\approx N$, we have
\[
    |\lambda_k^\alpha-j^\alpha|\lesssim N^{\alpha-1}.
\]
Hence the phases remain coherent for $|t|\lesssim N^{1-\alpha}$.
Consequently,
\[
    |e^{it\Delta^{\alpha/2}}f_j(x)|\gtrsim N^{\frac{d-1}{2}}~~
    \text{for}~~ |t|\lesssim N^{1-\alpha},~~ d_g(x,x_0)\lesssim N^{-1}.
\]
Taking an orthonormal family $(f_j)_j$ of cardinality comparable to $N$ with $j\approx N$ and setting $\nu_j=1$, we obtain
\[
    \Big\|\sum_j |e^{it\Delta^{\alpha/2}}f_j|^2\Big\|_{L^{p/2}_tL^{q/2}_x(I\times M)}
    \gtrsim N^{d}(N^{-d})^{\frac2q}N^{\frac{2}{p}(1-\alpha)}.
\]
Comparing this with the frequency-localized Strichartz estimate \eqref{eq:applocal} at regularity $s$,
we get
\[
    N^d(N^{-d})^{\frac2q}N^{\frac{2}{p}(1-\alpha)}
    \lesssim N^{2s}N^{1/\beta},
\]
which implies
\[
    \frac1\beta\ge d-\frac{2d}{q}+\frac{2(1-\alpha)}p-2s.
\]
Since
\(
    \gamma_\alpha(p,q)=\frac d2-\frac dq-\frac{\alpha}{p},
\)
it implies the necessary condition
\[
    \frac1\beta\ge \frac2p.
\]
This gives the necessary condition corresponding to the interior of $O A_{d/2}C_{d/2}D$ in the same long-time frequency-localized sense as in \cite{wang2025strichartz}.

The counting example must also be understood in this long-time sense.
Taking an orthonormal basis of the spectral subspace $\{\lambda_k\le N\}$, setting $\nu_j=1$, and using the interval length $|I|\approx N^{1-\alpha}$, the lower bound has size
\(
    N^d N^{\frac{2}{p}(1-\alpha)}.
\)
On the other hand, the frequency-localized estimate \eqref{eq:applocal} with regularity $s$ gives the upper bound $N^{2s}N^{d/\beta}$.
Thus
\[
    N^d N^{\frac{2}{p}(1-\alpha)}
    \lesssim N^{2s}N^{d/\beta},
\]
which implies
\[
    \frac1\beta
    \ge 1+\frac{2(1-\alpha)}{dp}-\frac{2s}{d}.
\]
Thus we have the necessary condition
\[
    \frac1\beta\ge \frac1{\beta_{d/2}(p,q)}.
\]
This gives the sharpness of the range of $\beta$ in the interior of $O A_{d/2}B$ for $0<\alpha<1$.

\subsection{Wave and Klein-Gordon case}
We finally consider $P=\sqrt{\Delta+m^2}$.
The spectral-cluster example above has an analogue in this case.
Indeed, for a fixed-width spectral cluster $\lambda_k\in(j-\delta,j]$ with $j\approx N$, we have
\[
    \left|\sqrt{\lambda_k^2+m^2}-\sqrt{j^2+m^2}\right|\lesssim 1.
\]
Therefore the corresponding phases remain coherent on a fixed time interval.
Using the same functions as in \eqref{eq:fj}, we obtain
\[
    |e^{it\sqrt{\Delta+m^2}}f_j(x)|\gtrsim N^{\frac{d-1}{2}}~~
    \text{for}~~ |t|\lesssim 1,
    ~~ d_g(x,x_0)\lesssim N^{-1}.
\]
Taking an orthonormal family $(f_j)_j$ of cardinality comparable to $N$, we have the lower bound
\[
    \Big\|\sum_j |e^{it\sqrt{\Delta+m^2}}f_j|^2\Big\|_{L^{p/2}_tL^{q/2}_x(I\times M)}
    \gtrsim N^d(N^{-d})^{\frac2q}.
\]
If the corresponding orthonormal Strichartz estimate \eqref{eq:applocal} holds with Sobolev regularity $s$, then the upper bound for such a frequency-$N$ system is
\[
    N^{2s}\|(1)_j\|_{\ell^\beta}
    \approx N^{2s}N^{1/\beta}.
\]
Thus
\[
    \frac1\beta\ge d-\frac{2d}{q}-2s.
\]
This gives the necessary condition
\[
    \frac1\beta\ge d-\frac{2d}{q}-2\gamma(p,q)=\frac{2}{p},
\]
since $\gamma(p,q)=\frac d2-\frac dq-\frac1p$.
Thus the range of $\beta$ in the interior of $O A_{(d-1)/2}C_{(d-1)/2}D$ in Theorem~\ref{ns2} is sharp on any compact manifold.

However, for the interior of the triangle $O A_{(d-1)/2}B$, the counting example only yields a weaker necessary condition.
Indeed, for the frequency-localized estimate \eqref{eq:applocal} at regularity $s$, taking an orthonormal basis of the spectral subspace $\{\lambda_k\le N\}$ and setting $\nu_j=1$ gives
\[
    N^d\lesssim N^{2s}N^{d/\beta},
\]
and hence
\[
    \frac1\beta\ge 1-\frac{2s}{d}.
\]
It can only give the condition
\[
    \frac1\beta\ge 1-\frac{2\gamma(p,q)}d
    =\frac1{\beta_{d/2}(p,q)}.
\]
This does not reach the threshold
\[
    \frac1{\beta_{(d-1)/2}(p,q)}=\frac2q+\frac{2}{(d-1)p},
\]
which appears in Theorem~\ref{ns2}.
Therefore the sharpness of the range of $\beta$ in the interior of $O A_{(d-1)/2}B$ remains open.
A similar issue also arises in the Euclidean setting.

\vskip2mm
\noindent\textbf{Acknowledgements.} 
AZ is partially supported by National Key R\&D Program of China No. 2024YFA1015300, Beijing Natural Science Foundation No. 1242009, National Natural Science Foundation of China No. 11801536, the China Scholarship Council No. 202506020208 and the Fundamental Research Funds for the Central Universities.  LX is supported by the National Natural Science Foundation of China No. 12101028 and the Fundamental Research Funds for the Central Universities. The authors would like to thank Cheng Zhang for inspiring discussions.

\noindent\textbf{Statements}  The authors have no relevant financial or non-financial interests to disclose. Data sharing is not applicable to this article as no datasets were used.

\bibliography{ns}
	
\bibliographystyle{alpha}

\end{document}